\documentclass[draft,reqno]{amsart}

\usepackage{amssymb}

\numberwithin{equation}{section}

\newcommand{\bbold}{\mathbb}

\def\Q { {\bbold Q} }
\def\Z { {\bbold Z} }
\def\C { {\bbold C} }
\def\N { {\bbold N} }

\def \order{\operatorname{order}}

\renewcommand\epsilon{\varepsilon}
\renewcommand\rho{\varrho}

\def \<{\langle}
\def \>{\rangle}

\def \tilde {\widetilde}

\def \supp {\operatorname{supp}}

\def \((  {(\!(}
\def \)) {)\!)}

\DeclareMathSymbol{\precequ}{\mathrel}{symbols}{"16}
\DeclareMathSymbol{\succequ}{\mathrel}{symbols}{"17}

\newtheorem{theorem}{Theorem}[section]
\newtheorem{lemma}[theorem]{Lemma}
\newtheorem{prop}[theorem]{Proposition}
\newtheorem{cor}[theorem]{Corollary}

\theoremstyle{definition}

\newtheorem{definition}[theorem]{Definition}

\theoremstyle{remark}
\newtheorem*{example}{Example}

\newtheorem{exampleNumbered}[theorem]{Example}

\newtheorem*{remark}{Remark}

\newtheorem*{conjecture}{Conjecture}

\newcommand{\abs}[1]{\lvert#1\rvert}
\newcommand{\dabs}[1]{\lVert#1\rVert}

\def \id{\operatorname{id}}

\def \bi{{\boldsymbol{i}}}
\def \bj{{\boldsymbol{j}}}

\renewcommand\leq{\leqslant}
\renewcommand\geq{\geqslant}

\renewcommand\frak{\mathfrak}

\DeclareMathAlphabet{\mathbf}{OML}{cmm}{b}{it}

\DeclareFontFamily{U}{fsy}{}
\DeclareFontShape{U}{fsy}{m}{n}{<->s*[.9]psyr}{}
\DeclareSymbolFont{der@m}{U}{fsy}{m}{n}
\DeclareMathSymbol{\der}{\mathord}{der@m}{182}

\DeclareSymbolFont{der@m}{U}{fsy}{m}{n}
\DeclareMathSymbol{\derdelta}{\mathord}{der@m}{100}

\newcommand\wt{\operatorname{wt}}

\newcommand\val{v}

\DeclareFontFamily{OMS}{smallo}{}
\DeclareFontShape{OMS}{smallo}{m}{n}{<->s*[.7]cmsy10}{}
\DeclareSymbolFont{smallo@m}{OMS}{smallo}{m}{n}
\DeclareMathSymbol{\smallo}{\mathord}{smallo@m}{79}

\DeclareFontFamily{OMS}{largerdot}{}
\DeclareFontShape{OMS}{largerdot}{m}{n}{<->s*[.8]cmsy10}{}
\DeclareSymbolFont{largerdot@m}{OMS}{largerdot}{m}{n}
\DeclareMathSymbol{\largerdot}{\mathord}{largerdot@m}{15}

\newcommand{\tr}{\operatorname{\mathfrak{tr}}}

\newcommand\diag{\operatorname{diag}}

\DeclareMathSymbol{\llambda}{\mathord}{der@m}{108}
\DeclareMathSymbol{\rrho}{\mathord}{der@m}{114}

\newcommand\itval{\operatorname{itval}}
\newcommand\itlog{\operatorname{itlog}}

\title[Logarithms of Iteration Matrices]{Logarithms of Iteration Matrices, and Proof of a Conjecture by Shadrin and Zvonkine}

\author{Matthias Aschenbrenner}
\address{University of California, Los Angeles, California, USA}
\email{matthias@math.ucla.edu}

\thanks{Partially supported by a grant from the National Science Foundation.}

\begin{document}

\begin{abstract}
A proof of a conjecture by Shadrin and Zvonkine, relating the entries of a matrix arising in the study of Hurwitz numbers to a certain sequence of rational numbers, is given. The main tools used are iteration matrices of formal power series and their (matrix) logarithms. %We also show that the exponential generating function of the sequence in question is differentially transcendental over the ring of convergent power series.
\end{abstract}

\maketitle

\footnotetext{November 2011}

\noindent
This note is devoted to the study of the somewhat mysterious-looking sequence
\begin{equation}\label{eq:mystery}\tag{S}
0, 1, -\frac{1}{2}, \frac{1}{2}, -\frac{2}{3}, \frac{11}{12}, -\frac{3}{4}, -\frac{11}{6}, \frac{29}{4}, \frac{493}{12}, -\frac{2711}{6}, -\frac{12406}{15}, \frac{2636317}{60}, \dots
\end{equation}
of rational numbers. I first encountered this sequence in ongoing joint work with van den Dries and van der Hoeven on asymptotic differential algebra \cite{AvdD}. It also appears in a conjecture made in a paper by Shadrin and Zvonkine \cite{SZ} in connection with a generating series for Hurwitz numbers (which count the number of ramified coverings of the sphere by a surface, depending on certain parameters like the degree of the covering and the genus of the surface). I came across \cite{SZ} by entering the numerators and denominators of the first few terms of \eqref{eq:mystery} into Sloane's {\it On-Line Encyclopedia of Integer Sequences}\/ \cite{OEIS}. (The numerator sequence is A134242, the denominator sequence is A134243.)
In this note we prove the conjecture from \cite{SZ}. In the course of doing so, we identify
a formula for the sequence  \eqref{eq:mystery}: denoting its $n$th term by $c_n$ (so $c_1=0$, $c_2=1$, $c_3=-\frac{1}{2}$ etc.), we have %\marginpar{Perhaps change to $s_j$}
$$c_{n}=\sum_{\substack{1\leq k<n \\ 1<n_1<\cdots<n_{k-1}<n_k=n}}  \frac{(-1)^{k+1}}{k}{n_2\brace n_1} {n_3\brace n_2}\cdots {n_k\brace n_{k-1}}.$$ 
Here and below, we denote by ${j\brace i}$ the Stirling numbers of the second kind:  ${j\brace i}$ is the number of equivalence relations on a $j$-element set with $i$ equivalence classes. They obey the recurrence relation
$$
{j\brace i} = {j-1\brace i-1}+i{j-1\brace i} \qquad (i,j>0)$$
with initial conditions
$${0\brace 0} = 1, \ {0\brace i} = {j\brace 0} = 0 \qquad (i,j>0).$$
For example, we have
\begin{alignat*}{3}
1-\frac{1}{2}{3\brace 2} &= 1-\frac{1}{2}\cdot 3&=\,-\frac{1}{2}=c_3  &&  \\
1-\frac{1}{2}\left({4\brace 2}+{4\brace 3}\right)+\frac{1}{3}{3\brace 2}{4\brace 3} &= 
 1-\frac{1}{2}(7+6)+\frac{1}{3}\cdot 3\cdot 6 & =
\frac{1}{2}=c_4 && \\
\left.\begin{aligned}
\textstyle 1-\frac{1}{2}\left({5\brace 2}+{5\brace 3}+{5\brace 4}\right) + \    \\ 
\textstyle \frac{1}{3}\left({3\brace 2}{5\brace 3}+{4\brace 2}{5\brace 4}+{4\brace 3}{5\brace 4}\right) -  \ \\ \textstyle \frac{1}{4}{3\brace 2}{4\brace 3}{5\brace 4}\quad\ \end{aligned} \right\} &=
\left\{ \begin{aligned}
&1-\textstyle\frac{1}{2}(15+25+10)+ \\
&\textstyle\frac{1}{3}(3\cdot 25+7\cdot 10+6\cdot 10)- \\  &\textstyle \frac{1}{4}\cdot 3\cdot 6\cdot 10\end{aligned}\right\}
&=-\frac{2}{3} =c_5 && 
\end{alignat*}
A key concept for our study of \eqref{eq:mystery} is the {\it iteration matrix}\/ of a formal power series; these matrices are well-known in the iteration theory of analytic functions \cite{J1,J2} and in combinatorics \cite{Comtet-Book}. The iteration matrix of a power series $f\in\Q[[z]]$ of the form $f=z+z^2g$ ($g\in\Q[[z]]$) is a certain bi-infinite upper triangular matrix with rational entries associated to $f$.
After stating the conjecture of Shadrin and Zvonkine in Section~\ref{sec:conj} and making some preliminary reductions, we summarize some general definitions and basic facts about  triangular matrices in Section~\ref{sec:triangular} and introduce the group of iteration matrices in Section~\ref{sec:iteration}. In Section~\ref{sec:lie} we determine its Lie algebra of infinitesimal generators, by slightly generalizing results of Schippers \cite{Schippers}. These results tie in with a notion from classical iteration theory:
the infinitesimal generator of the iteration matrix of a formal power series $f$ as above is uniquely determined by another power series $\itlog(f)\in z^2\Q[[z]]$, introduced by Jabotinsky \cite{J2} and called the {\it iterative logarithm}\/  of $f$ by \'Ecalle \cite{Ecalle2}. Some of the properties of iterative logarithms are discussed in Section~\ref{sec:itlog}, before we return to the proof of the conjecture of Shadrin-Zvonkine in Section~\ref{sec:proof}. The exponential generating function (egf) of the sequence $(c_n)$, that is, the formal power series 
$$\sum_{n\geq 1} c_n\frac{z^n}{n!}=\frac{1}{2}z^2-\frac{1}{12}z^3+\frac{1}{48}z^4-\frac{1}{180}z^5+\cdots,$$ 
turns out to be nothing else than the iterative logarithm of the power series $e^z-1$. 

The iterative logarithm $\itlog(f)$ of any formal power series $f$ satisfies a certain
functional equation found by Jabotinsky \cite{J1}. In the case of $f=e^z-1$, this equation leads to a convolution formula for Stirling numbers (and another formula for the terms of the sequence $(c_n)$):
\begin{multline}\label{eq:C}\tag{C}
c_n=\sum_{\substack{1\leq k<n\\ 1<n_1<\cdots<n_{k-1}<n_k=n}}  \frac{(-1)^{k+1}}{k}{n_2\brace n_1} {n_3\brace n_2}\cdots {n_k\brace n_{k-1}} =  \\
\sum_{\substack{1\leq k<n-1\\ 1<n_1<\cdots<n_{k-1}<n_k=n-1}} \frac{(-1)^{k}}{k+1}{n_2\brace n_1} {n_3\brace n_2}\cdots {n_k\brace n_{k-1}}
\end{multline}
To our knowledge, this formula does not seem to have been noticed before.
(For instance, it does not appear in Gould's collection of combinatorial identities \cite{Gould}.) We give a proof of \eqref{eq:C} in Section~\ref{sec:proof}.

Shadrin and Zvonkine write that the sequence \eqref{eq:mystery} {\it seems to be quite irregular}\/ \cite[p.~224]{SZ}.
This impression can be substantiated as follows. A formal power series $f\in\C[[z]]$ is said to be {\it differentially algebraic}\/  if it satisfies an algebraic differential equation, i.e., an equation
$$P(z,f,f',\dots,f^{(n)})=0$$
where $P$ is a non-zero polynomial in $n+2$ indeterminates with constant complex coefficients.
The coefficient sequence $(f_n)$ of every differentially algebraic power series $f=\sum_{n\geq 0} f_nz^n\in \Q[[z]]$ is regular in the sense that it satisfies a certain kind of (generally non-linear) recurrence relation \cite[pp.~186--194]{Mahler}. A class of differentially algebraic power series which is of particular importance in combinatorial enumeration is the class of {\it $D$-finite}\/ (also called {\it holonomic}\/) power series \cite[Chapter~6]{Stanley-II}. These are the series  whose coefficient sequence satisfies a homogeneous {\it linear}\/ recurrence relation of finite degree with polynomial coefficients. Equivalently  \cite[Proposition~6.4.3]{Stanley-II} a formal power series $f\in\C[[z]]$ is $D$-finite if and only if $f$ satisfies a non-trivial {\it linear}\/ differential equation 
$$a_0f+a_1f'+\cdots+a_nf^{(n)}=0\qquad (a_i\in\C[z],\ a_n\neq 0).$$ % i.e., if one can choose $P$ as above of the form  $P=a_0Y_0+\cdots+a_nY_n-b$ where $a_i,b\in \C[z]$. 
(This class includes, e.g., all hypergeometric series.)
In Section~\ref{sec:proof} we will see that the egf of $(c_n)$ is not differentially algebraic. 
This is a consequence of a result of Boshernitzan and Rubel, stated without proof in \cite{BR}, which characterizes when the iterative logarithm of a power series satisfies an ADE; in Section~\ref{sec:difftr} below we give a complete proof of this fact.
It is also known  \cite{Baker-58, Lewin} that the egf of $(c_n)$ has radius of convergence $0$.
Indeed, a common generalization of these results holds true: {\it the egf of $(c_n)$ does not satisfy an algebraic differential equation over the ring of convergent power series.}\/ The proof of this fact will be given elsewhere \cite{A}. It seems likely (though we have not investigated this further) that the {\it ordinary}\/ generating function (ogf) 
$$\sum_{n\geq 1} c_nz^n=z^2-\frac{1}{2}z^3+\frac{1}{2}z^4-\frac{2}{3}z^5+\cdots$$ 
of the sequence \eqref{eq:mystery} is also differentially transcendental. (Note, however, that there are examples of sequences of rationals whose egf is differentially transcendental yet whose ogf is differentially algebraic; see \cite[Proposition~6.3~(i)]{LR}.)

% is given in the final Section~\ref{sec:final} of this paper. 

\subsubsection*{Notations and conventions}
We let $d$, $m$, $n$, $k$, possibly with decorations, range over $\N=\{0,1,2,\dots\}$. All rings below are assumed to have a unit $1$. Given a ring $R$ we denote by $R^\times$ the group of units of $R$.

\subsubsection*{Acknowledgements}
We thank the anonymous referees whose corrections and suggestions improved the paper.

\section{The Conjecture of  Shadrin and Zvonkine}\label{sec:conj}

\noindent
Before we can formulate this conjecture, we need to fix some notation.
Let $K$ be a commutative ring and let $R=K[[t_0,t_1,\dots]]$ be the ring of powers series in
the pairwise distinct indeterminates $t_0,t_1,\dots$, with coefficients from $K$. We equip $R$ with the $\frak m$-adic topology, where $\frak m$ is the ideal $(t_0,t_1,\dots)$ of $R$.
In this subsection we let $\bi$, $\bj$ range over 
the set of sequences $\bi=(i_0,i_1,\dots)\in\N^\N$ such that $i_n=0$ for all but finitely many $n$. For each $\bi$ we set
$$t^{\bi} := t_0^{i_0}t_1^{i_1} \cdots t_n^{i_n}\cdots\in R.$$
Hence every element $f$ of $R$ can be uniquely written in the form
$$f = \sum_{\bi} f_{\bi}\, t^{\bi}\qquad\text{where $f_{\bi}\in K$ for all $\bi$.}$$
We call an element of $R$ of the form $at^\bi$, where $0\neq a\in K$, a monomial.
We put
$$\dabs{\bi}  := 1i_0+ 2i_1 + 3i_2 + \cdots + (n+1)i_n + \cdots \in\N,$$
%\begin{align*}
%\abs{\bi}  & := i_0+ i_1 + i_2 + \cdots + i_n + \cdots\in\N \\
%\dabs{\bi} & := 1i_0+ 2i_1 + 3i_2 + \cdots + (n+1)i_n + \cdots \in\N,
%\end{align*}
and we define a valuation $v$ on $R$ by setting   
$$\val(f):=\min_{f_\bi\neq 0}\,\dabs{\bi}\in\N\text{ for $0\neq f\in R$,} \qquad
\val(0):=\infty>\N.$$ 
%Note that for $f,g\in R$ we have
%$$\val(f+g)\geq\min\{\val(f),\val(g)\},$$
%and if $K$ (and hence $R$) is an integral domain, then
%$$\val(f\cdot g)=\val(f)+\val(g).$$
%so $\abs{i}\leq\dabs{\bi}$.
%We call $\deg(t^\i):=\abs{\bi}$ the degree of the monomial $t^\bi$, and $\val(t^\i):=\dabs{\bi}$ its weight. For non-zero $f\in R$ we also set 
Suppose from now on that $K=\Q[z]$ where $z$ is a new indeterminate over $\Q$.
Shadrin and Zvonkine first introduce rational numbers $a_{d,d+k}$  by the equation
\begin{equation}\label{eq:equation for a}
\sum_{b=1}^{d+1}
{d\choose b-1} \frac{(-1)^{d-b+1}}{d!}
\cdot \frac{1}{1-b\psi}= \sum_{k\geq 0} a_{d,d+k}\psi^{d+k}
\end{equation}
in the formal power series ring $\Q[[\psi]]$:
\begin{align*}
\frac{1}{1-\psi} &= 1+\psi+\psi^2+\cdots & &(d=0) \\
-\frac{1}{1-\psi}+\frac{1}{1-2\psi} &= \psi+3\psi^2+7\psi^3+\cdots &&(d=1) \\
\frac{1/2}{1-\psi}-\frac{1}{1-2\psi}+\frac{1/2}{1-3\psi} &= \psi^2+6\psi^3+25\psi^4+\cdots & &(d=2) \\
&\vdots
\end{align*}
Using the numbers $a_{d,d+k}$ (which turn out to be positive integers, see Lemma~\ref{lem:stirling} below) they then define a sequence $(L_k)_{k>0}$ of differential operators on $R$: abbreviating the $K$-derivation $\frac{\partial}{\partial t_n}$ of 
$R$ by $\partial_n$,  set
$$L_k = \sum_{\substack{0\leq r\leq k \\ k_1+\cdots+k_r=k\\ k_1,\dots,k_r>0\\ n_1,\dots,n_r\geq 0}} 
\frac{1}{r!}\, a_{n_1,n_1+k_1}\cdots a_{n_r,n_r+k_r}\, t_{n_1+k_1}\cdots t_{n_r+k_r}\, \partial_{n_1}\cdots\partial_{n_r} \qquad (k>0).$$
Note that the definition of $L_k$ (as a $K$-linear map $R\to R$) makes sense, since for every $\bi$, either
$$t_{n_1+k_1}\cdots t_{n_l+k_r}\, \partial_{n_1}\cdots\partial_{n_r}(t^\bi)$$
is zero or is a monomial which has valuation $\dabs{\bi}+k_1+\cdots+k_r$
and which is divisible by $t_{n_1+k_1}\cdots t_{n_r+k_r}$; moreover,
given $\bj$ there are only finitely many $\bi$ with $\dabs{\bi}<\dabs{\bj}$,
and only finitely many $k_1,\dots,k_r>0$ and $n_1,\dots,n_r\geq 0$
such that $j_{n_1+k_1},\dots,j_{n_r+k_r}>0$.
The first few terms of the sequence $(L_k)$ are
\begin{align*}
L_1 &= \sum_{n_1} a_{n_1,n_1+1} t_{n_1+1}\,\partial_{n_1}, \\
L_2 &= \sum_{n_1} a_{n_1,n_1+2} t_{n_1+2}\,\partial_{n_1} + \frac{1}{2!} \sum_{n_1,n_2} a_{n_1,n_1+1}a_{n_2,n_2+1}\, t_{n_1+1}t_{n_2+1}\partial_{n_1}\partial_{n_2} \\
L_3 &= \sum_{n_1} a_{n_1,n_1+3} t_{n_1+3}\,\partial_{n_1} + \frac{1}{2!} \sum_{n_1,n_2} a_{n_1,n_1+1}a_{n_2,n_2+2}\, t_{n_1+1}t_{n_2+2}\partial_{n_1}\partial_{n_2} + \\
&\qquad \frac{1}{2!} \sum_{n_1,n_2} a_{n_1,n_1+2}a_{n_2,n_2+1}\, t_{n_1+2}t_{n_2+1}\partial_{n_1}\partial_{n_2} +  \\
&\qquad \frac{1}{3!} \sum_{n_1,n_2,n_3} a_{n_1,n_1+1}a_{n_2,n_2+1}a_{n_3,n_3+1}\, t_{n_1+1}t_{n_2+1}t_{n_3+1}\partial_{n_1}\partial_{n_2}\partial_{n_3},
\end{align*}
and in general we have
\begin{equation}\label{eq:Lk general}
L_k=  \sum_{n_1} a_{n_1,n_1+k} t_{n_1+k}\,\partial_{n_1} + \text{higher-order operators} \qquad (k>0).
\end{equation}
To streamline the notation we set $L_0:=\id_R$. The argument above shows that for every $f\in R$ we have $\val(L_k(f))\geq k+\val(f)$, hence  the sequence $(z^kL_k(f))_k$ is summable in $R$. Thus one may
combine the $L_k$ to a $K$-linear map $\mathbf L \colon R\to R$ with
$$\mathbf L(f)=\sum_k z^kL_k(f)=f+zL_1(f)+z^2L_2(f)+\cdots\qquad\text{for all $f\in R$.}$$
The operator $\mathbf L$ is used in \cite{SZ} to perform a change of variables in a certain formula for Hurwitz numbers coming from \cite{ELSV}. The following proposition is established in \cite[Proposition~A.8]{SZ}. (The formula for $l_k$ given in \cite{SZ} mistakenly omits the summation over $n$.)

\begin{prop}\label{prop:SZ}
There are rational numbers $\alpha_{n,n+k}$ such that, setting
$$l_k  = \sum_{n} \alpha_{n,n+k}t_{n+k}\,\partial_n \qquad (k>0)$$
and
$$\mathbf l = z l_1 + z^2 l_2 + \cdots,$$
we have $\mathbf L=\exp(\mathbf l)$, i.e., 
\begin{equation}\label{eq:L=exp(l)}
\mathbf L(f) = \sum_{n} \frac{1}{n!} {\mathbf l}^n(f)\qquad\text{for every $f\in R$.}
\end{equation}
\end{prop}

(To see that the definition of $l_k$ and $\mathbf l$ makes sense argue as for $L_k$ and $\mathbf L$ above; since 
$\val(\mathbf l(f))\geq \val(f)+1$ we have
$\val(\mathbf l^n(f))\geq \val(f)+n$ for all $n$, hence the sum on the right-hand side of the equation in \eqref{eq:L=exp(l)} exists in $R$.)

\medskip

After proving this proposition, Shadrin and Zvonkine make the following conjecture about the form of the $\alpha_{n,n+k}$. (Again, we correct a typo in \cite{SZ}: in Conjecture~A.9 replace $t_n\frac{\partial}{\partial t_{n+k}}$ by $t_{n+k}\frac{\partial}{\partial t_n}$.)

\begin{conjecture}
For all $k>0$  and all $n$,
$$\alpha_{n,n+k} = c_{k+1} \binom{n+k+1}{k+1}$$
where $(c_k)_{k\geq 1}$ is a sequence of rational numbers, with the first terms  given by \eqref{eq:mystery}.
\end{conjecture}

The first step in our proof of this conjecture is to realize is that the $a_{d,d+k}$ are essentially the Stirling numbers of the second kind. We extend the definition of $a_{d,d+k}$ by setting $a_{dd}:=1$ for every $d$.

\begin{lemma}\label{lem:stirling}
For every $d$ and $k$,
$$a_{d,d+k} = {d+k+1 \brace d+1}.$$
\end{lemma}
\begin{proof}
We expand the left-hand side of \eqref{eq:equation for a} in powers of $\psi$:
$$
\sum_{b=1}^{d+1}
{d\choose b-1} \frac{(-1)^{d-b+1}}{d!}
\cdot \frac{1}{1-b\psi} = 
  \sum_{i\geq 0} \left( \frac{1}{d!} \sum_{b=1}^{d+1} (-1)^{d-b+1} \binom{d}{b-1} b^i \right) \psi^i.
$$
Now we focus on the coefficient of $\psi^i$ in the last sum. By the Binomial Theorem, this coefficient can be written as
$$\frac{1}{d!} \sum_{b=0}^d (-1)^{d-b} \binom{d}{b} (b+1)^i = 
\sum_{j=0}^i \binom{i}{j} \left( \frac{1}{d!} \sum_{b=0}^d (-1)^{d-b} \binom{d}{b} b^j\right).$$
It is well-known that 
$${j\brace d} = \frac{1}{d!} \sum_{b=0}^d (-1)^{d-b} \binom{d}{b} b^j$$
and
$$\sum_{j=0}^i \binom{i}{j} {j\brace d} = {i+1\brace d+1}.$$
(See, e.g., identities (6.19) respectively (6.15) in \cite{GKP}.) The lemma follows.
\end{proof}

By \eqref{eq:Lk general} and the above lemma  we therefore have 
$$L_k(t_d) = a_{d,d+k}t_{d+k}={d+k+1\brace d+1}t_{d+k} $$ 
and hence
\begin{equation}\label{eq:Ltn}
\mathbf L(t_d) = \sum_k {d+k+1\brace d+1}  z^k t_{d+k}.
\end{equation}
Moreover, by definition of $l_k$ we have $l_k(t_d)=\alpha_{d,d+k}t_{d+k}$ for all $d$ and $k>0$, hence
$$l(t_d) = \sum_{k>0} \alpha_{d,d+k} z^k t_{d+k}$$
and thus for every $n>0$:
$$l^n(t_d) =  \sum_{k_1,\dots,k_n>0}  \alpha_{d,d+k_1}\cdots \alpha_{d+k_1+\cdots+k_{n-1},d+k_1+\cdots+k_n} z^{k_1+\cdots+k_n} t_{d+k_1+\cdots+k_n}.
$$
This yields
$$\exp(\mathbf l)(t_d)=\sum_k \left(
\sum_{\substack{k_1+\cdots+k_n=k\\ n > 0,\ k_1,\dots,k_n>0}} \frac{1}{n!}\alpha_{d,d+k_1}\cdots \alpha_{d+k_1+\cdots+k_{n-1},d+k}\right) z^k t_{d+k}$$
and therefore, by \eqref{eq:Ltn} and Proposition~\ref{prop:SZ}:
\begin{equation}\label{eq:add+k}
{d+k+1\brace d+1} = \sum_{\substack{k_1+\cdots+k_n=k\\ n>0,k_1,\dots,k_n>0}} \frac{1}{n!}\alpha_{d,d+k_1}\alpha_{d+k_1,d+k_1+k_2}\cdots \alpha_{d+k_1+\cdots+k_{n-1},d+k}.
\end{equation}
It is suggestive to express this equation as an identity between matrices.
We define ${j\brace i}:=0$ for $i>j$, and
combine the Stirling numbers of the second kind into a bi-infinite upper triangular matrix:
\begin{equation}\label{eq:Stirling}
S=(S_{ij})=
\begin{pmatrix}
1 & 0 & 0  & 0  & 0  & 0   & \cdots \\
  & 1 &  1 & 1  & 1  & 1   & \cdots \\
  &   & 1  & 3  & 7  & 15  & \cdots \\
  &   &    & 1  &  6 & 25  & \cdots \\
  &   &    &    &  1 &  10 & \cdots \\
  &   &    &    &    & 1   & \cdots \\
  &   &    &    &    &     & \ddots  
\end{pmatrix} \qquad\text{where $S_{ij}={j\brace i}$.}
\end{equation}
We also introduce the upper triangular matrix
$$A=(\alpha_{ij})=
\begin{pmatrix}
0 & 1 & -\frac{1}{2}	& \frac{1}{2}  	& -\frac{2}{3}	& \frac{11}{12}	& \cdots \\
  & 0 & 3 			& -2  			& \frac{5}{2}  	& -4   			& \cdots \\
  &   & 0  			& 6  			& -5  			& \frac{15}{2}	& \cdots \\
  &   &    			& 0  			&  10 			& 10  			& \cdots \\
  &   &    			&    			&  0 			& -15 			& \cdots \\
  &   &    			&    			&    			& 0   			& \cdots \\
  &   &    			&    			&    			&     			& \ddots  
\end{pmatrix}\qquad\text{where $\alpha_{ij}:=0$ for $i\geq j$.}$$ 
Then \eqref{eq:add+k} may be written as
\begin{alignat*}{2}
S_{i+1,j+1} & =  \sum_{\substack{k_1+\cdots+k_n=j-i\\ n>0,\ k_1,\dots,k_n>0}} \frac{1}{n!}\,\alpha_{i,i+k_1}\alpha_{i+k_1,i+k_1+k_2}\cdots \alpha_{i+k_1+\cdots+k_{n-1},j} \\
			& = \sum_{n=1}^{j-i} \frac{1}{n!}\, (A^n)_{ij} & \quad (i\leq j)
\end{alignat*}
or equivalently, writing $S^+:=(S_{i+1,j+1})_{i,j}$ and employing the matrix exponential:
$$S^+ = \sum_{n\geq 0} \frac{1}{n!} A^n = \exp(A)$$
Therefore, in order to prove the conjecture from \cite{SZ}, we need to be able to express the matrix logarithm of $S^+$ in some explicit manner. We show how this can be done (and finish the proof of the conjecture) in Section~\ref{sec:proof} below; before that, we need to step back and first embark on a systematic study of a class of matrices (iteration matrices) which encompasses $S$ and many other matrices of combinatorial significance (Sections~\ref{sec:triangular} and \ref{sec:iteration}), and of their matrix logarithms (Sections~\ref{sec:lie} and \ref{sec:itlog}).

\section{Triangular Matrices}\label{sec:triangular}

\noindent
In this section we let $K$ be a commutative ring. 

\subsection*{The $K$-algebra of triangular matrices}
We construe $K^{\N\times\N}$ as a $K$-module with the componentwise addition and scalar multiplication. The elements $M=(M_{ij})_{i,j\in\N}$ of $K^{\N\times\N}$ may be visualized as
bi-infinite matrices with entries in $K$:
$$M = \begin{pmatrix}
M_{00} & M_{01} & M_{02} &  \cdots &\\
M_{10} & M_{11} & M_{12} &  \cdots &\\
M_{20} & M_{21} & M_{22} &  \cdots &\\
\vdots & \vdots & \vdots &  \ddots &
\end{pmatrix}.$$
We say that $M=(M_{ij})\in K^{\N\times\N}$ is (upper) {\bf triangular} if $M_{ij}=0$ for all $i,j\in\N$ with $i>j$. We usually write a triangular matrix $M$ in the form
$$M = 
\begin{pmatrix}
M_{00} 	& M_{01} 	& M_{02} 	& M_{03}  & \cdots  \\
  	& M_{11}				& M_{12}	& M_{13} & \cdots  \\
	&				& M_{22}			& M_{23} & \cdots  \\
	&				&				& M_{33}			  & \cdots  \\
	&				&				&			& \ddots
\end{pmatrix}.$$
Given triangular matrices
$M=(M_{ij})$ and $\tilde{M}=(\tilde{M}_{ij})$, the product 
$$M\cdot\tilde{M} := \left(\textstyle\sum_{k} M_{ik}\tilde{M}_{kj}\right)_{i,j\in\N}$$
makes sense and is again a triangular matrix.
Equipped with this operation, the $K$-submodule of $K^{\N\times\N}$ consisting of all triangular matrices becomes an associative $K$-algebra $\tr_K$ with unit $1$ given by the identity matrix. If $K$ is a subring of a commutative ring $L$, then $\tr_K$ is a $K$-subalgebra of the $K$-algebra $\tr_L$. 
We also define 
$$[M,N] := MN-NM \qquad\text{for $M,N\in\tr_K$}.$$
Then the $K$-module $\tr_K$ equipped with the binary operation $[\hskip0.5em , \hskip0.5em ]$ is a Lie $K$-algebra.

\medskip

For every $n$ we set
$$\tr_K^n := \big\{ M=(M_{ij})\in\tr_K : \text{$M_{ij}=0$ for all $i,j\in\N$ with $i-j+n\geq 1$}\big\}.$$
%So $\tr_{K,0}=\tr_K$, and $\tr_K^1$ consists of all strictly triangular matrices.
We call the elements of $\tr^1_K$ {\bf strictly triangular.} 
It is easy to verify that the sequence $(\tr^n_K)$ of $K$-submodules of $\tr_K$ is a filtration of the $K$-algebra $\tr_K$, i.e.,
\begin{enumerate}
\item $\tr_K^0=\tr_K$;
\item $\tr_K^n\supseteq \tr_K^{n+1}$ for all $n$; 
\item $\tr_K^m \tr_K^n\subseteq \tr_K^{m+n}$ for all $m$, $n$; and
\item $\bigcap_{n} \tr_K^n=\{0\}$.
\end{enumerate}
Clearly $\tr_K$ is complete in the topology making $\tr_K$ into a topological ring with fundamental system of neighborhoods of $0$ given by the $\tr_K^n$.

\medskip

The group $\tr_K^\times$ of units of $\tr_K$ has the form
$$\tr_K^\times = D_K\ltimes (1+\tr_K^1) \qquad\text{(internal semidirect product of subgroups of $\tr_K^\times$)}$$
where $D_K$ is the group of diagonal invertible matrices:
$$D_K := \big\{ M=(M_{ij})\in \tr_K : 
\text{$M_{ii}\in K^\times$ and $M_{ij}=0$ for $i\neq j$}\big\}.$$

\subsection*{Diagonals}
We say that a matrix $M=(M_{ij})\in\tr_K$ is {\bf $n$-diagonal} if $M_{ij}=0$ for $j\neq i+n$. We simply call $M$ {\bf diagonal} if $M$ is $0$-diagonal.
Given a sequence $a=(a_i)_{i\geq 0}\in K^\N$, we denote by $\diag_n a$ the $n$-diagonal matrix $M=(M_{ij})\in K^{\N\times\N}$ with 
$M_{i,i+n}=a_i$ for every $i$.  %We also abbreviate $\diag_0 a$ as $\diag a$.
The sum of two $n$-diagonal matrices is $n$-diagonal. As for products, we have:

\begin{lemma}\label{lem:formulas for diagonals}
Let $M=\diag_m a$ be $m$-diagonal and $N=\diag_n b$ be $n$-diagonal, where $a=(a_i), b=(b_i)\in K^\N$. Then $M\cdot N$ is $(m+n)$-diagonal, in fact
$$M\cdot N = \diag_{m+n}( a_i\cdot b_{i+m} )_{i\geq 0}.$$
Therefore $[M,N]$ is $(m+n)$-diagonal, with
$$[M,N] = \diag_{m+n}( a_i\cdot b_{i+m} - b_i\cdot a_{i+n} )_{i\geq 0},$$
and for each $k$, the matrix $M^k$ is $km$-diagonal, with
$$M^k = \diag_{km} (a_i \cdot a_{i+m} \cdots a_{i+(k-1)m} )_{i\geq 0}.$$
\end{lemma}

\subsection*{Exponential and logarithm of triangular matrices}
In this subsection we assume that $K$ contains $\Q$ as a subring. Then
for each strictly triangular matrix $M$, the sequences 
$\big(\frac{M^n}{n!}\big)_{n\geq 0}$ and
$\big((-1)^{n+1}\frac{M^n}{n} \big)_{n\geq 1}$ 
are summable, and the maps
$$\tr_K^1\to 1+\tr_K^1\colon M\mapsto \exp(M) := \sum_{n\geq 0} \frac{M^n}{n!}$$
and
$$1+\tr_K^1 \to \tr_K^1\colon M\mapsto \log(M) := \sum_{n\geq 1} (-1)^{n+1}\frac{(M-1)^n}{n}$$
are mutual inverse; in particular, they are bijective. If $M\in\tr_K^n$, $n>0$, then $\exp(M)\in 1+\tr_K^{n}$ and $\log(1+M)\in\tr^n_K$.
It is easy to see that 
\begin{equation}\label{eq:exp for commuting matrices}
\exp(M)\exp(N)=\exp(M+N)\qquad\text{for all $M,N\in \tr_K^1$ with $MN=NM$.}
\end{equation}
In particular
$$\exp(M)^k=\exp(kM)\qquad\text{for all $M\in\tr_K^1$, $k\in\Z$.}$$
We also note that given a unit $U$ of $\tr_K$, we have
$$%\begin{equation}\label{eq:conjugation exponential}
\exp(UMU^{-1})=U\exp(M)U^{-1}\qquad\text{for all $M\in\tr_K^1$}
$$%\end{equation}
and
\begin{equation}\label{eq:conjugation logarithm}
\log\big(UMU^{-1}\big)=U\log(M)U^{-1}\qquad\text{for all $M\in 1+\tr_K^1$.}
\end{equation}
Given $M=(M_{ij})_{i,j}\in\tr_K$ we define $M^+:=(M_{i+1,j+1})_{i,j}\in\tr_K$. It is easy to see that $M\mapsto M^+$ is a $K$-algebra morphism $\tr_K\to\tr_K$ with $M\in\tr_K^n\Rightarrow M^+\in\tr_K^n$. Thus, for $M\in\tr_K^1$:
\begin{equation}\label{eq:M+}
\exp(M^+)=\exp(M)^+,\qquad \log(1+M^+)=\log(1+M)^+.
\end{equation}
From Lemma~\ref{lem:formulas for diagonals} we immediately obtain, for all $M=\diag_1 a$ where $a=(a_i)\in K^\N$: 
\begin{equation}\label{eq:exp of diag1}
(\exp M)_{ij} = \frac{1}{(j-i)!}\, a_i\cdot a_{i+1} \cdots a_{j-1} \qquad\text{for all $i,j\in\N$ with $i\leq j$.}
\end{equation}

\subsection*{Derivations on the $K$-algebra of triangular matrices}
Let $\der$ be a derivation of $K$, i.e., a map $\der\colon K\to K$ such that
$$\der(a+b)=\der(a)+\der(b),\quad \der(ab)=\der(a)b+a\der(b)\qquad\text{for all $a,b\in K$.}$$
Given $M=(M_{ij})\in\tr_K$ we let
$$\der(M):=(\der(M_{ij}))\in\tr_K.$$
Then $M\mapsto\der(M)\colon\tr_K\to\tr_K$ is a derivation of $\tr_K$, i.e.,
$$\der(M+N)=\der(M)+\der(N),\quad \der(MN)=\der(M)N+M\der(N)\qquad\text{for all $M,N\in\tr_K$.}$$ 
Note that $\der(\tr^n_K)\subseteq\tr^n_K$ for every $n$.

\medskip

We now let $t$ be an indeterminate over $K$, and we work in the polynomial ring $K^*=K[t]$ and in the $K^*$-algebra $\tr_{K^*}$ (which contains $\tr_K$ as a $K$-subalgebra). We equip $K^*$ with the derivation $\frac{d}{dt}$.
The following two elementary observations are used in Section~\ref{sec:lie}. 
Until the end of this subsection we assume that $K$ contains $\Q$ as a subring.

\begin{lemma}\label{lem:derivative of exp(tM)}
Let $M\in\tr^1_K$. Then
$$\frac{d}{dt} \exp(tM) = \exp(tM)M.$$
\end{lemma}
\begin{proof}
We have $(tM)^n=t^nM^n$ for every $n$, hence
$$\exp(tM)=\sum_{n\geq 0} \frac{(tM)^n}{n!}=\sum_{n\geq 0} \frac{t^n M^n}{n!}$$
and thus
$$\frac{d}{dt}\exp(tM) = \sum_{n\geq 0} \frac{d}{dt}\left(\frac{t^n M^n}{n!}\right) = \sum_{n>0} \frac{t^{n-1}M^n}{(n-1)!} = \exp(tM)M.$$
(Similarly, of course, one also sees $\frac{d}{dt} \exp(tM) = M\exp(tM)$, but we won't need this fact.)
\end{proof}

The following lemma is a familiar fact about homogeneous systems of linear differential equations with constant coefficients:

\begin{lemma}\label{lem:linear DE}
Let $M,Y_0\in\tr^1_K$ and $Y\in\tr^1_{K^*}$. Then
$$\frac{dY}{dt}=YM\text{ and } Y\big\lvert_{t=0} = Y_0 \qquad\Longleftrightarrow\qquad Y=Y_0\exp(tM).$$
\end{lemma}
\begin{proof}
Lemma~\ref{lem:derivative of exp(tM)} shows that if $Y=Y_0\exp(tM)$ then $\frac{dY}{dt}=YM$, and clearly
$Y\big\lvert_{t=0}=Y_0\exp(0)=Y_0$. Conversely, suppose $\frac{dY}{dt}=YM$ and $Y\big\lvert_{t=0} = Y_0$. Then $Y_1:=Y-Y_0\exp(tM)\in\tr^1_{K^*}$ satisfies $\frac{dY_1}{dt}=Y_1M$ and $Y_1\big\lvert_{t=0} = 0$; hence after replacing $Y$ by $Y_1$ we may assume that $\frac{dY}{dt}=YM$ and $Y\big\lvert_{t=0} = 0$, and need to show that then $Y=0$. For a contradiction suppose $Y\neq 0$, and
write $Y=(Y_{ij})$ where $Y_{ij}\in K^*$ and $M=(M_{ij})$ where $M_{ij}\in K$. Since  $Y\big\lvert_{t=0} = 0$, for each $i$, $j$ such that $Y_{ij}\neq 0$ we can write
$Y_{ij}=t^{n_{ij}}Z_{ij}$ with $n_{ij}\in\N$, $n_{ij}>0$, and $Z_{ij}\in K^*$, $Z_{ij}(0)\neq 0$.
Choose $i$, $j$ so that $n_{ij}$ is minimal. Then by $\frac{dY}{dt}=YM$ we have
$$n_{ij}t^{n_{ij}-1}Z_{ij}+t^{n_{ij}}\frac{dZ_{ij}}{dt}=\frac{dY_{ij}}{dt} = \sum_k Y_{ik} M_{kj} = \sum_{Y_{ik}\neq 0} t^{n_{ik}} Z_{ik} M_{kj}, $$
thus
$$Z_{ij}=\frac{1}{n_{ij}} \left( -t\frac{dZ_{ij}}{dt} + \sum_{Y_{ik}\neq 0} t^{n_{ik}-n_{ij}+1} Z_{ik} M_{kj} \right)$$
and hence $Z_{ij}(0)=0$, a contradiction. So $Y=0$ as desired.
\end{proof}

\section{Iteration Matrices}\label{sec:iteration}

\noindent
Let $K$ be a commutative ring containing $\Q$ as a subring. 
%\begin{equation}\label{eq:expgen}
%f(t) = \sum_{n\geq 0} f_n\frac{t^n}{n!}\in K[[t]].
%\end{equation}
%Note that then
%$$f_n =\frac{d^n}{dt^n}f\bigg\lvert_{t=0}.$$
Let $A=\Q[y_1,y_2,\dots]$ where $(y_n)_{n\geq 1}$ is a sequence of pairwise distinct indeterminates, let $z$ be an indeterminate distinct from each $y_n$, and let 
$$y = \sum_{n\geq 1} y_n\frac{z^n}{n!}\in A[[z]].$$
Then, with $x$ another new indeterminate, we have in the power series ring $A[[x,z]]$:
\begin{equation}\label{eq:Bell polynomials}
\exp(x\cdot y) = \sum_{n\geq 0} \frac{(x\cdot y)^n}{n!} = \sum_{i,j\in\N} B_{ij} x^{i} \frac{z^j}{j!}
\end{equation}
where $B_{ij}=B_{ij}(y_1,y_2,\dots)$ are polynomials in $\Q[y_1,y_2,\dots]$, known as the {\bf Bell polynomials.} A general reference for properties of the $B_{ij}$ is Comtet's book \cite{Comtet-Book}.
(Our notation slightly differs from the one used in \cite{Comtet-Book}: $B_{ij}=\text{\bf B}_{ji}$.) We can obtain $B_{ij}$ by differentiating \eqref{eq:Bell polynomials} appropriately and setting $x=z=0$:
$$B_{ij}=\frac{1}{i!} \frac{\partial^i\partial^j}{\partial x^i\partial z^j}\exp(x\cdot y)\bigg\lvert_{x=z=0} 
        =\frac{1}{i!} \frac{d^j}{dz^j} y^i \bigg\lvert_{z=0},$$
hence
$$%\begin{equation}\label{eq:power of y, factorials}
\frac{1}{i!} y^i = \sum_{j\geq 0} B_{ij} \frac{z^j}{j!}.
$$%\end{equation}
In particular, we immediately see that $B_{0j}=0$  and
$B_{1j}=y_j$ for $j\geq 1$. Since 
$$\frac{1}{i!} y^i = y_1^i\frac{z^i}{i!}+\text{terms of higher degree (in $z$)}$$
we also see that $B_{ij}=0$ whenever $i>j$ and $B_{jj}=y_1^j$ for all $j$.
It may also be shown (see \cite[Section~3.3, Theorem~A]{Comtet-Book}) that $B_{ij}\in\Z[y_1,\dots,y_{j-i+1}]$, and $B_{ij}$ is homogeneous of degree $i$ and isobaric of weight $j$. (Here each $y_j$ is assigned weight~$j$.)
Given a power series $f\in z K[[z]]$, written in the form
$$f=\sum_{n\geq 1} f_n\frac{z^n}{n!}\qquad\text{($f_n\in K$ for each $n\geq 1$),}$$
we now define the triangular matrix
\begin{multline*}
[f] := \big([f]_{ij}\big)_{i,j\in\N}=\big(B_{ij}(f_1,f_2,\dots,f_{j-i+1})\big)_{i,j\in\N}= \\ 
\begin{pmatrix}
1 & 0   & 0     & 0   		& 0   			& 0   					& \cdots \\
  & f_1 & f_2   & f_3 		& f_4 			& f_5					& \cdots \\
  &     & f_1^2 & 3f_1f_2 	& 4f_1f_3+3f_2^4	& 5f_1f_4+10f_2f_3		& \cdots \\
  &     &       & f_1^3      & 6f_1^2f_2		& 10f_1^2f_3+15f_1f_2^2	& \cdots \\
  &     &       &            & f_1^4         & 10f_1^3f_2				& \cdots \\
  &		&		&			&				& f_1^5					& \cdots \\
  &		&		&			&				&						&\ddots
\end{pmatrix}\in\tr_K.
\end{multline*}
More generally, suppose $\Omega=(\Omega_n)$ is a {\bf reference sequence,} i.e., a sequence of non-zero rational numbers with $\Omega_0=\Omega_1=1$. Then we define the Bell polynomials with respect to $\Omega$ by setting
$$y = \sum_{n\geq 1} y_n\Omega_n\,z^n\in A[[z]]$$
and expanding
\begin{equation}\label{eq:power of y}
\Omega_i y^i = \sum_{j\geq 0} B^\Omega_{ij} \Omega_j\, z^j
\end{equation}
where $B^\Omega_{ij}=B^\Omega_{ij}(y_1,y_2,\dots)\in\Q[y_1,y_2,\dots]$.
As above, one sees that  $B^\Omega_{0j}=0$  and
$B^\Omega_{1j}=y_j$ for $j\geq 1$, as well as $B^\Omega_{ij}=0$ whenever $i>j$ and $B^\Omega_{jj}=y_1^j$ for all $j$.
For
$$f=\sum_{n\geq 1} f_n\,\Omega_n z^n\in zK[[z]]\qquad\text{($f_n\in K$ for each $n\geq 1$),}$$
we define
$$[f]^\Omega := \big([f]_{ij}^\Omega\big)_{i,j\in\N}\in\tr_K\qquad\text{where 
$[f]^\Omega_{ij}:=B^\Omega_{ij}(f_1,f_2,\dots,f_{j-i+1})$.}$$
Thus, denoting the reference sequence $(1/n!)$ by $\Phi$,
we have $B^\Phi_{ij}=B_{ij}$ for each $i$, $j$ and $[f]^\Phi=[f]$ for each $f\in zK[[z]]$.
Note that by \eqref{eq:power of y} we have, for all reference sequences $\Omega$, $\tilde\Omega$:
\begin{equation}\label{eq:convert iteration matrices}
\frac{\Omega_j}{\Omega_i}\,[f]^\Omega_{ij}=\frac{\tilde{\Omega}_j}{\tilde{\Omega}_i}\,[f]^{\tilde\Omega}_{ij}\qquad\text{for all $i$, $j$,}
\end{equation}
that is,
\begin{equation}\label{eq:convert iteration matrices, 2}
(D^\Omega)^{-1}\,[f]^\Omega\,D^\Omega = (D^{\tilde{\Omega}})^{-1}\,[f]^{\tilde\Omega}\,D^{\tilde\Omega}
\end{equation}
where $D^\Omega$ is the diagonal matrix
$$D^\Omega=\begin{pmatrix}
\Omega_0     	&				&				&			\\
				& \Omega_1    	&				&			\\
				&				& \Omega_2   	&			\\
				&				&				& \ddots
\end{pmatrix}\in\tr_\Q^\times.$$
In particular, for every reference sequence $\Omega$ we have, with  $\mathbf 1$ denoting the constant sequence $(1,1,1,\dots)$:
\begin{equation}\label{eq:conjugate iteration matrix}
[f]^\Omega = D^\Omega(D^\Phi)^{-1}\,[f]\,D^\Phi (D^\Omega)^{-1} = D^\Omega\,[f]^{\mathbf 1}\,(D^\Omega)^{-1}.
\end{equation}
As first noticed by Jabotinsky \cite{J1,J2}, a crucial property of $[\ \ ]^\Omega$ is that it converts composition of power series into matrix multiplication  \cite[Section~3.7, Theorem~A]{Comtet-Book}:
\begin{equation}\label{eq:Jabotinsky}
[f\circ g]^\Omega=[f]^\Omega\cdot [g]^\Omega\qquad\text{for all $f,g\in zK[[z]]$.}
\end{equation}
To see this, repeatedly use \eqref{eq:power of y} to obtain
\begin{multline*}
\sum_{j\geq 0} [f\circ g]^\Omega_{ij}\, \Omega_j z^j 	=  \Omega_i (f\circ g)^i
										=  \Omega_i f^i\circ g 
										=  \sum_{k\geq 0} [f]^\Omega_{ik}\, \Omega_k g^k \\
										=  \sum_{k\geq 0} [f]^\Omega_{ik} \sum_{j\geq 0} [g]^\Omega_{kj}\, \Omega_j z^j 
										= \sum_{j\geq 0} \left(\sum_{k\geq 0} [f]^\Omega_{ik}[g]^\Omega_{kj}\right) \Omega_j z^j 
\end{multline*}
and compare the coefficients of $z^j$.
The matrix $[f]^\Omega$ is called the {\bf iteration matrix of $f$} with respect to $\Omega$ in \cite{Comtet-Book}. 
(To be precise, \cite{Comtet-Book} uses the transpose of our $[f]^\Omega$.)
For $[f]$, the term {\it convolution matrix}\/ of $f$ is also in use (cf.~\cite{Knuth}), and $[f]^{\mathbf 1}$ is called the {\it power matrix}\/ of $f$ in \cite{Schippers}.

The subset $zK^\times + z^2K[[z]]$ of $zK[[z]]$ forms a group under composition (with identity element $z$), and $f\mapsto [f]^\Omega$ restricts to an embedding of this group into the group $\tr_K^\times$ of units of $\tr_K$. 
(In particular, $[z]^\Omega=1$ for each $\Omega$.) As in \cite{Comtet-Book}, we say that $f\in zK[[z]]$ is {\bf unitary} if $f_1=1$. The set of unitary power series in $K[[z]]$ is a subgroup of $zK^\times + z^2K[[z]]$ under composition, whose image under  $f\mapsto [f]^\Omega$ 
is a subgroup of $1+\tr_K^1$ which we denote by $\mathcal M_K^\Omega$. If $\Omega$ is clear from the context, we simply write $\mathcal M_K=\mathcal M_K^\Omega$. By \eqref{eq:conjugate iteration matrix}, the matrix groups $\mathcal M_K^\Omega$, for varying $\Omega$, are all conjugate to each other. We call $\mathcal M_K^\Omega$ the {\bf group of iteration matrices over $K$} with respect to $\Omega$. 

Given $f\in K[[z]]$ of the form $f=z+z^{n+1}g$ with $n>0$ and $g\in K[[z]]$ such that $g(0)\neq 0$, we say that the {\bf iterative valuation} of $f$ is $n$; in symbols: $n=\itval(f)$. (See \cite{Ecalle2}.)
It is easy to see that for $f\in zK[[z]]$ and $n>0$, we have $f\in z+z^{n+1}K[[z]]$ if and only if $[f]^\Omega\in 1+\tr_K^n$.
For each $n>0$ we define the subgroup
$$\mathcal M_K^{\Omega,n} := \mathcal M_K^{\Omega}\cap (1+\tr_K^n) = \big\{ [f]^\Omega: f\in z+z^{n+1}K[[z]] \big\}$$
of $\mathcal M_K^\Omega$. Then 
$$\mathcal M_K^\Omega=\mathcal M_K^{\Omega,1}\supseteq \mathcal M_K^{\Omega,2}\supseteq\cdots\supseteq \mathcal M_K^{\Omega,n}\supseteq\cdots\quad\text{and}\quad \bigcap_{n>0} \mathcal M_K^{\Omega,n}=\{1\},$$ 
and if $f\in zK[[z]]$ is unitary with $f\neq z$, then $n=\itval(f)$ is the unique $n>0$ such that $[f]^\Omega\in\mathcal M_K^{\Omega,n}\setminus \mathcal M_K^{\Omega,n+1}$.

As shown by Erd{\H{o}}s and Jabotinsky \cite{EJ},
iteration matrices can be used to define ``fractional'' iterates of formal power series. 
%Given a power series $f\in K[[z]]$ we denote by $f^{[n]}$ the $n$th iterate of $f$, defined inductively by $f^{[0]}:=z$ and $f^{[n+1]}=f\circ f^{[n]}$.
Let $t$ be a new indeterminate and $K^*=K[t]$.

\begin{prop}[Erd{\H{o}}s and Jabotinsky]\label{prop:EJ}
Suppose $K$ is an integral domain, and
let $f\in zK[[z]]$ be unitary. Then there exists a unique power series $f^{[t]}\in zK^*[[z]]$ such that, writing $f^{[a]}:=f^{[t]}\big|_{t=a}\in zK[[z]]$ for $a\in K$:
\begin{enumerate}
\item $f^{[0]}=z$;
\item $f^{[a+1]}=f^{[a]}\circ f$ for all $a,b\in K$.
\end{enumerate}
The power series $f^{[t]}$ is given by
$$f^{[t]}=\sum_{j\geq 1} M_{1j}\frac{z^j}{j!}\quad\text{where } M := \sum_{n\geq 0} {t\choose n} \big([f]-1\big)^n\in\tr_{K^*}.$$
Here for every $n$ as usual
${t\choose n} = \frac{1}{n!}t(t-1)\cdots(t-n+1)\in\Q[t]$.
\end{prop}
\begin{proof}
%We denote by $f^{\circ n}$ the $n$th iterate of $f$, defined inductively by $f^{\circ 0}:=z$ and $f^{\circ(n+1)}=f\circ f^{\circ n}$.
Since $[f]-1\in\tr_K^1$, the sum defining $M$
%$$B := \sum_{n\geq 0} {t\choose n} \big([f]-1\big)^n$$
exists in $\tr_{K^*}$, and $M\big|_{t=n}=[f]^n$ for every $n$, by the binomial formula. % and \eqref{eq:Jabotinsky}.
Let 
$f^{\circ t}:=\sum_{j\geq 1} M_{1j}\frac{z^j}{j!}$,
and for an element $a$ in a ring extension of $K^*$ write $f^{\circ a}:=f^{\circ t}\big|_{t=a}$. 
Then
$[f^{\circ n}]_{1j}=M_{1j}\big|_{t=n}=([f]^n)_{1j}$ for every $j\geq 1$ and thus $f^{\circ n}$ is the $n$th iterate of $f$: $f^{\circ n}=f\circ f\circ\dots \circ f$ ($n$ times). In particular $f^{\circ 1}=f$ and
$f^{\circ(m+n)}=f^{\circ m}\circ f^{\circ n}$ for all $m$, $n$.
Hence
if $s$ is another indeterminate, then $f^{\circ (s+t)}=f^{\circ s}\circ f^{\circ t}$ (in $K[s,t][[z]]$), since the coefficients (of equal powers of $z$) of both sides of this equation are polynomials in $s$ and $t$ with coefficients in $K$ which agree for all integral values of $(s,t)$. This shows that $f^{\circ t}$ satisfies conditions (1) and (2) (with $f^{\circ\cdot}$ replacing $f^{[\,\cdot\,]}$ everywhere). If $f^{[t]}\in K^*[[z]]$ is any power series satisfying (1) and (2), then $f^{[n]}=f^{\circ n}$ is the $n$th iterate of $f$, for every $n$, and as before we deduce $f^{[t]}=f^{\circ t}$.
\end{proof}

The power series $f^{[a]}$ ($a\in K$) in this proposition form a subgroup of $zK[[z]]$ under composition which contains $f$; they may be thought of as ``fractional iterates'' of $f$. (This explains the choice of the term ``iteration matrix.'') 

\medskip

Some examples of iteration matrices are collected below. Many more (in the case where $\Omega=\Phi$) are given in \cite{Knuth}.

\begin{example} 
Suppose $f=\frac{z}{1-z}$. Then $$[f]_{ij}={j-1\choose i-1}\frac{j!}{i!}\in\N\qquad (i>0)$$ are the {\bf Lah numbers}; here and below we set ${j\choose i}:=0$ for $i>j$.
(See \cite[Section~3.3, Theorem~B]{Comtet-Book}.)
Thus if
$\Omega_n=\frac{1}{n}$ for each $n>0$, then by \eqref{eq:convert iteration matrices}
$$[f]^\Omega_{ij}=\frac{\Omega_i\Phi_j}{\Omega_j\Phi_i}[f]_{ij}={j\choose i}\qquad\text{for $i>0$,}$$ 
hence 
\begin{equation}\label{eq:binom}
[f]^\Omega = \begin{pmatrix}
1 & 0 & 0  & 0  & 0  & 0   & \cdots \\
  & 1 & 2  & 3  & 4  & 5   & \cdots \\
  &   & 1  & 3  & 6  & 10  &\cdots \\
  &   &    & 1  & 4  & 10  & \cdots \\
  &   &    &    &  1 & 5   & \cdots \\
  &   &    &    &    & 1   & \cdots \\
  &   &    &    &    &     & \ddots 
\end{pmatrix}\in\tr_\Z
\end{equation}
is Pascal's triangle of binomial coefficients (except for the first row).
\end{example}

\begin{example}
The Stirling numbers of the second kind have the egf
$$e^{x(e^z-1)}=\sum_{i,j} {j\brace i} x^i \frac{z^j}{j!},$$ 
cf.~\cite[Section~1.14,~(III)]{Comtet-Book} or \cite[(7.54)]{GKP}. Hence by \eqref{eq:Bell polynomials} we have
\begin{equation}\label{eq:Stirling itmatrix}
[e^z-1]=S,
\end{equation}
where $S$ is as in \eqref{eq:Stirling}.
The matrix $S$ is a unit in $\tr_\Z$, and it is well-known (see \cite[Section~3.6~(II)]{Comtet-Book}) that the entries of its inverse
\begin{equation}\label{eq:Sinverse}
S^{-1}=(S^{-1}_{ij})=\begin{pmatrix}
1 & 0 & 0  & 0  & 0  & 0   & \cdots \\
  & 1 & -1 & 2  & -6 & 24  & \cdots \\
  &   & 1  & -3 & 11 & -50 & \cdots \\
  &   &    & 1  & -6 & 35  & \cdots \\
  &   &    &    &  1 & -10 & \cdots \\
  &   &    &    &    & 1   & \cdots \\
  &   &    &    &    &     & \ddots  
\end{pmatrix}
\end{equation}
are the signed Stirling numbers of the first kind: $S^{-1}_{ij}=(-1)^{j-i}{j\brack i}$, where $j\brack i$ denotes the 
number of permutations of a $j$-element set having $i$ disjoint cycles.
Thus \eqref{eq:Jabotinsky} and \eqref{eq:Stirling itmatrix} yields
$\big[\log(1+z)\big]=S^{-1}$.
\end{example}

\section{The Lie Algebra of the Group of Iteration Matrices}\label{sec:lie}

\noindent
Throughout this section we let $K$ be a commutative ring which contains $\Q$ as a subring. We let $\Omega$ denote a reference sequence.
We need a description of the Lie algebra of the matrix group $\mathcal M_K=\mathcal M_K^\Omega$, generalizing the one of the Lie algebra of $\mathcal M_\C^{\mathbf 1}$ from \cite{Schippers}. The arguments follow \cite{Schippers}, except that we replace the complex-analytic ones used there by algebraic ones.

\begin{definition}
Let $h=\sum_{n} h_n z^{n}\in zK[[z]]$. The {\bf infinitesimal iteration matrix} of $h$ with respect to $\Omega$ is the triangular matrix
\begin{multline*}
\langle h\rangle^\Omega = \left(\langle h\rangle^\Omega_{ij}\right) = 
 \begin{pmatrix}
0 & 0  & 0    & 0    & 0    & \cdots \\
  & h_1& \frac{\Omega_1}{\Omega_2}\hskip0.5em h_2  & \frac{\Omega_1}{\Omega_3}\hskip0.5em h_3  & \frac{\Omega_1}{\Omega_4}\hskip0.5em h_4  & \cdots \\
  &   & \hskip1em 2h_1 & \frac{\Omega_2}{\Omega_3}2h_2 & \frac{\Omega_2}{\Omega_4}2h_3 &\cdots \\
  &   &    & \hskip1em 3h_1 & \frac{\Omega_3}{\Omega_4}3h_2 & \cdots \\
  &   &    &      & \hskip1em 4h_1 & \cdots \\
  &   &    &      &      & \ddots 
\end{pmatrix}\in\tr_K\\ \text{where $\langle h\rangle^\Omega_{ij}=\frac{\Omega_i}{\Omega_j}ih_{j-i+1}$.}
\end{multline*}
\end{definition}

Note that if $\Omega$, $\tilde\Omega$ are reference sequences, then
\begin{equation}\label{eq:conjugate h}
(D^\Omega)^{-1} \langle h\rangle^\Omega D^{\Omega} = (D^{\tilde\Omega})^{-1} \langle h\rangle^{\tilde\Omega} D^{\tilde\Omega},
\end{equation}
in particular
$$\langle h\rangle^\Omega = D^\Omega(D^\Phi)^{-1}\,\langle h\rangle\,D^\Phi (D^\Omega)^{-1} = D^\Omega\,\langle h\rangle^{\mathbf 1}\,(D^\Omega)^{-1}.$$

\begin{exampleNumbered}\label{ex:factorials}
For $h=\sum_{n} h_n z^{n}\in zK[[z]]$ we have
\begin{multline*}
\langle h\rangle :=\langle h\rangle^\Phi =
 \begin{pmatrix}
0 & 0  & 0    & 0    & 0    & \cdots \\
  & h_1& \frac{2!}{1!}\hskip0.5em h_2  & \frac{3!}{1!}\hskip0.5em h_3  & \frac{4!}{1!}\hskip0.5em h_4  & \cdots \\
  &   & \hskip1em 2h_1 & \frac{3!}{2!}2h_2 & \frac{4!}{2!}2h_3 &\cdots \\
  &   &    & \hskip1em 3h_1 & \frac{4!}{3!}3h_2 & \cdots \\
  &   &    &      & \hskip1em 4h_1 & \cdots \\
  &   &    &      &      & \ddots 
\end{pmatrix} \\
\text{where $\langle h\rangle_{ij}=\frac{j!}{(i-1)!}h_{j-i+1}$ for $i>0$.}
\end{multline*}
\end{exampleNumbered}

For each $n$ we have $h\in z^{n+1} K[[z]]$ if and only if $\langle h\rangle^\Omega \in \tr_K^n$.
We define the $K$-submodule
$$\frak m_K^{\Omega,n} :=  \big\{ \langle h\rangle^\Omega: h\in z^{n+1}K[[z]]\big\}$$
of $\tr_K^n$, and we set $\frak m_K^\Omega := \frak m_K^{\Omega,1}$; so 
$$\frak m_K^\Omega= \frak m_K^{\Omega,1}\supseteq  \frak m_K^{\Omega,2}\supseteq\cdots\supseteq  \frak m_K^{\Omega,n}\supseteq\cdots\quad\text{and}\quad\bigcap_{n>0}  \frak m_K^{\Omega,n}=\{0\}.$$ 
If $\Omega$ is clear from the context, we abbreviate $\frak m_K=\frak m_K^\Omega$ and $\frak m_K^n=\frak m_K^{\Omega,n}$.
We set
$$e_n^\Omega := \langle z^{n+1} \rangle^\Omega,$$
and we write $e_n$ if the reference sequence $\Omega$ is clear from the context.
The matrix $e_n=e_n^\Omega$ is $n$-diagonal; in fact
$$e_n = \diag_n \left(\textstyle\frac{\Omega_i}{\Omega_{i+n}}i\right)\in\frak m_K^{n}.$$
Clearly the infinitesimal iteration matrix with respect to $\Omega$ of a power series from $zK[[z]]$ can be uniquely written  as an infinite sum
$$h_1 e_0 + h_2 e_1 + \cdots\qquad\text{where $h_n\in K$ for every $n>0$.}$$
Using Lemma~\ref{lem:formulas for diagonals} one verifies easily that 
$$[e_m,e_n]=(m-n)e_{m+n}\qquad\text{for all $m$, $n$.}$$
This implies that 
$$\frak m_K^n=Ke_n+ Ke_{n+1}+ \cdots\qquad (n>0)$$ 
is an ideal of the Lie $K$-algebra $\tr^1_K$.
The main goal of this section is to show the following generalization of a result of Schippers \cite{Schippers}:

\begin{theorem}\label{thm:Schippers}
Let $n>0$. Then $\exp(\frak m_K^n) = \mathcal M_K^n$ \textup{(}and hence $\log(\mathcal M_K^n)=\frak m_K^n$\textup{)}.
\end{theorem}

\begin{exampleNumbered}\label{ex:embeddable}
Let $f=\frac{z}{1-z}\in z\Q[[z]]$, and suppose $\Omega_n=\frac{1}{n}$ for every $n>0$. Then by \eqref{eq:exp of diag1} and \eqref{eq:binom} one sees easily that
$$\log\,[f]^\Omega = \diag_1(0,2,3,4,\dots) = 
\begin{pmatrix}
0 & 0 & 0  &    &    & \cdots \\
  & 0 & 2  & 0  &    & \cdots \\
  &   & 0  & 3  & 0  & \cdots \\
  &   &    & 0  & 4  & \cdots \\
  &   &    &    & 0  & \cdots \\
  &   &    &    &    & \ddots 
\end{pmatrix}=\langle z^2\rangle^\Omega\in\frak m_\Q^1.$$
\end{exampleNumbered}

We give the proof of this theorem after some preparatory results.
Below we let $t$ be a new indeterminate and $K^*=K[t]$.

\begin{lemma}\label{lem:matrix differential}
Let $f\in zK^*[[z]]$ and $h\in zK[[z]]$ satisfy
$$\frac{\partial f}{\partial t} = \frac{\partial f}{\partial z}\, h.$$
Then
$$\frac{d}{dt} [f]^{\Omega} = [f]^{\Omega}\langle h\rangle^{\Omega}.$$
\end{lemma}
\begin{proof}
We need to show that for all $i$ we have
$$\frac{d}{dt}[f]^{\Omega}_{ij}=\big([f]^{\Omega}\langle h\rangle^{\Omega}\big)_{ij}\qquad\text{for each $j$.}$$
For $i=0$ this is an easy computation, so suppose $i>0$. We have
$$\frac{\partial f^i}{\partial z}=\sum_{j\geq 1} j[f]^{\Omega}_{ij}\Omega_j\,z^{j-1}$$
and hence
$$\frac{\partial f^i}{\partial z}\,h=\sum_{j\geq 0} \left(\sum_{k=1}^j k[f_{ik}]^{\Omega}\Omega_k h_{j-k+1}\right) z^j.$$
Moreover
$$\frac{\partial f^i}{\partial t}=\sum_{j\geq 0} \frac{d}{dt} [f]^{\Omega}_{ij}\Omega_j\,z^j.$$
By the hypothesis of the lemma
$$
\frac{\partial f^i}{\partial t} = if^{i-1}\frac{\partial f}{\partial t} = if^{i-1} \frac{\partial f}{\partial z}\, h=
\frac{\partial f^i}{\partial z}\,h,$$
hence
$$\frac{d}{dt}[f]^{\Omega}_{ij} = \sum_{k=1}^{j} k[f_{ik}]^{\Omega}\frac{\Omega_k}{\Omega_j}h_{j-k+1}=\big([f]^{\Omega}\langle h\rangle^{\Omega}\big)_{ij}$$
for each $j$ as required. 
\end{proof}

This lemma is used in the proof of the following important proposition:

\begin{prop}\label{prop:Schippers}
Let $h\in z^{n+1}K[[z]]$, where $n>0$, and set
$$f_t := \sum_{j\geq 1}\, (\exp\,t\langle h\rangle^\Omega)_{1j}\Omega_j\, z^j \in z+z^{n+1}K^*[[z]].$$
Then
\begin{equation}\label{eq:Loewner}
\frac{\partial f_t}{\partial t} = \frac{\partial f_t}{\partial z}\,h
\end{equation}
and hence
\begin{equation}\label{eq:Loewner, 2}
[f_t]^{\Omega} = \exp\,t\langle h\rangle^\Omega.
\end{equation}
\end{prop}
\begin{proof}
By Lemma~\ref{lem:derivative of exp(tM)} we have
$$\frac{d}{dt} \exp\,t\langle h\rangle^\Omega = (\exp\,t\langle h\rangle^\Omega)\,\langle h\rangle^{\Omega}.$$
Hence
\begin{align*}
\frac{\partial f_t}{\partial t}	&= 
\sum_{j\geq 1} \, \left(\frac{d}{dt}\exp\,t\langle h\rangle^\Omega\right)_{1j}\Omega_j\, z^j \\
&= \sum_{j\geq 1} \, \big((\exp\,t\langle h\rangle^\Omega)\langle h\rangle^\Omega\big)_{1j}\Omega_j\, z^j \\
&= \sum_{j\geq 1} \left( \sum_{i=1}^j (\exp\,t\langle h\rangle^\Omega)_{1i}\langle h\rangle^\Omega_{ij}\Omega_j\right) z^j  \\
&= \sum_{j\geq 1} \left(\sum_{i=1}^j (\exp\,t\langle h\rangle^\Omega)_{1i}ih_{j-i+1}\Omega_i\right) z^j = \frac{\partial f_t}{\partial z}\,h. 
\end{align*}
By Lemma~\ref{lem:matrix differential} this yields $\frac{d}{dt} [f_t]^{\Omega} = [f_t]^{\Omega}\langle h\rangle^{\Omega}$. This shows that both $Y=[f_t]^\Omega$ and $Y=\exp\,t\langle h\rangle^\Omega$ satisfy $\frac{dY}{dt}=Y \,\langle h\rangle^{\Omega}$ and $Y\big\lvert_{t=0} =1$.
Hence $[f_t]^{\Omega} = \exp\,t\langle h\rangle^\Omega$ by Lemma~\ref{lem:linear DE}.
\end{proof}

The equation \eqref{eq:Loewner} is called the {\it formal Loew\-ner partial differential equation}\/ in \cite{Schippers}.
The following corollary, obtained by setting $t=1$ in \eqref{eq:Loewner, 2} above, shows in particular that $\exp(\mathfrak m_K^n)\subseteq\mathcal M_K^n$ for each $n>0$:

\begin{cor}\label{cor:Schippers}
Let $h\in z^{n+1}K[[z]]$, where $n>0$, and set
$$f := \sum_{j\geq 1}\, (\exp\,\langle h\rangle^\Omega)_{1j}\Omega_j\, z^j \in z+z^{n+1}K[[z]].$$
Then $[f]^\Omega=\exp\, \langle h\rangle^\Omega$.
\end{cor}

%We first note that 
%\begin{multline*}
%f^{\mathbf 1}	= \sum_{j\geq 1} (\exp\,\langle h\rangle^{\mathbf 1})_{1j} z^j 
%				= \sum_{j\geq 1} \left(\Omega_1 (\exp\,\langle h\rangle^{\mathbf 1})_{1j} \Omega_j^{-1}\right) \Omega_j z^j \\
%				= \sum_{j\geq 1} (\exp\,\langle h\rangle^\Omega)_{1j}\Omega_j\,z^j = f^\Omega
%\end{multline*}
%where in the third equation we used \eqref{eq:conjugate h}.				 Thus by \eqref{eq:conjugate iteration matrix} we have
%$$[f^\Omega]^\Omega = [f^{\mathbf 1}]^\Omega = D^\Omega\,[f^{\mathbf 1}]^{\mathbf 1}\,(D^\Omega)^{-1}.$$
%Since
%$$D^\Omega\,\exp\,\langle h\rangle^{\mathbf 1}\,(D^\Omega)^{-1} = \exp\big(D^\Omega\langle h\rangle^{\mathbf 1}(D^\Omega)^{-1}\big)=\exp\, \langle h\rangle^{\Omega}$$
%by \eqref{eq:conjugation exponential} and \eqref{eq:conjugate h}, we see that is is enough to prove Proposition~\ref{prop:Schippers} in the case where $\Omega=\mathbf 1$.

As above we write $e_k=e_k^\Omega$.
Given $k_1,\dots,k_n$ and $k=k_1+\cdots+k_n$, we have
$$e_{k_1} \cdots e_{k_n} = \diag_{k} \left(\textstyle\frac{\Omega_i}{\Omega_{i+k}}\,i(i+k_1)(i+k_1+k_2)\cdots (i+k_1+\cdots+k_{n-1})\right)_{i\geq 0}$$
by Lemma~\ref{lem:formulas for diagonals}.
Now let $M:=\langle h\rangle^{\Omega}$ where $h\in zK[[z]]$. So
$$M=\langle h\rangle^\Omega=h_1 e_0 + h_2 e_1 + \cdots$$
and hence
$$M^n = \sum_{k_1,\dots,k_n} h_{k_1+1}\cdots h_{k_n+1}\, e_{k_1}\cdots e_{k_n},$$
that is,
\begin{equation}\label{eq:Mn}
(M^n)_{ij} = \sum_{k_1+\cdots+k_n=j-i} h_{k_1+1}\cdots h_{k_n+1}\, \textstyle\frac{\Omega_i}{\Omega_{j}}\,i(i+k_1)\cdots (i+k_1+\cdots+k_{n-1})
\end{equation}
for all $i$, $j$. This observation leads to:

\begin{lemma}\label{lem:3.10 corrected}
Suppose $n>0$. Then 
$$(M^n)_{11} = h_1^n,\qquad (M^n)_{1j} = \frac{j^n-1}{\Omega_j(j-1)}h_1^{n-1}h_{j} + P^\Omega_{nj}(h_1,\dots,h_{j-1})\quad\text{for $j\geq 2$,}$$
where $P^\Omega_{nj}(Y_0,\dots,Y_{j-2})\in\Q[Y_0,\dots,Y_{j-2}]$ is homogeneous of degree $n$ and isobaric of weight $j-1$, and independent of $h$. \textup{(}Here each $Y_i$ is assigned weight $i$.\textup{)}
\end{lemma}
\begin{proof}
Set $i=1$ in \eqref{eq:Mn}. Then the only terms involving $h_{j}$ in this sum are those of the form
$h_1^{n-1}h_{j} \,\frac{1}{\Omega_{j}}\,j^{n-m}$ where $m\in\{1,\dots,n\}$.
This yields the lemma.
\end{proof}

An analogue of the preceding lemma (for $K=\C$ and $\Omega=\mathbf 1$) is Lemma~3.10 of \cite{Schippers}; however, the formula given there is wrong:

\begin{example}
Suppose $h=h_1z+h_2z^2$ and $\Omega=\mathbf 1$. Then 
$$M = \langle h\rangle^{\mathbf 1} = \begin{pmatrix}
0 & 0  & 0    & 0     & 0 & \\
  & h_1&  h_2  & 0  & 0 & \ddots \\
  &   &  2h_1 & 2h_2 & 0 &\ddots \\
  &   &    & 3h_1    & 3h_2 &\ddots \\
  &   &    &       & 4h_1 & \ddots \\
  &   &    &      &   &\ddots 
\end{pmatrix}$$
and hence
$$M^2 = \begin{pmatrix}
0 & 0  & 0    & 0     & 0 & \\
  & h_1^2&  3h_1h_2  & 2h_2^2  & 0 & \ddots \\
  &   &  4h_1^2 & 10h_1h_2 & 6h_2^2 &\ddots \\
  &   &    & 9h_1^2    & 21h_1h_2 &\ddots \\
  &   &    &       & 16h_1^2 & \ddots \\
  &   &    &      &   &\ddots 
\end{pmatrix}.$$
According to \cite[Lemma~3.10]{Schippers} we should have, for $j\geq 2$:
$$(M^2)_{1j}=2h_1h_{j}+\text{polynomial in $h_1,\dots,h_{j-1}$.}$$
However $(M^2)_{12}=3h_1h_2$ is not of this form. 
\end{example}

In the proof of Theorem~\ref{thm:Schippers} we are concerned with the case where $h\in z^2K[[z]]$, for which we need a refinement of Lemma~\ref{lem:3.10 corrected}:

\begin{lemma}
Suppose $h\in z^2K[[z]]$ and $n>0$. Then 
$$(M^n)_{1j} = 
\begin{cases} 
\frac{1}{\Omega_j}h_{j}		&\text{if $n=1$,} \\
P^\Omega_{nj}(h_1,\dots,h_{j-1})	&\text{if $1<n<j$,} \\
0								&\text{otherwise.}
\end{cases}$$
\end{lemma}
\begin{proof}
We have $h_1=0$, hence if $n>1$ then $(M^n)_{1j} = P^\Omega_{nj}(h_1,\dots,h_{j-1})$ by the previous lemma. 
We have $M\in\tr^1_K$ and hence $M^n\in\tr^n_K$, so $(M^n)_{1j}=0$ if $j-1<n$, that is, if $j\leq n$. The lemma follows.
\end{proof}

\begin{cor}
Suppose $h\in z^2K[[z]]$. Then for $j\geq 2$:
$$(\exp M)_{1j} = \frac{1}{\Omega_j}h_{j} + P^\Omega_j(h_2,\dots,h_{j-1})$$
where  $P^\Omega_{j}(Y_1,\dots,Y_{j-2})\in\Q[Y_1,\dots,Y_{j-2}]$ is  independent of $h$. \textup{(}In particular, $(\exp M)_{1j}$ is polynomial in $h_2,\dots,h_{j}$.\textup{)} Moreover, $P^\Omega_2=0$, and for $j>2$, $P^\Omega_j$ has degree $j-1$ and is isobaric of weight $j-1$.
\end{cor}
\begin{proof}
By the previous lemma we have
$$(\exp M)_{1j} = \sum_{n=1}^{j-1} \frac{1}{n!} (M^n)_{1j} = \frac{1}{\Omega_j}h_{j} + \sum_{n=2}^{j-1} \frac{1}{n!}P^\Omega_{nj}(h_1,\dots,h_{j-1}).$$
Hence $$P^\Omega_j(Y_1,\dots,Y_{j-2}):=\sum_{n=2}^{j-1} \frac{1}{n!}P^\Omega_{nj}(0,Y_1,\dots,Y_{j-2})$$ has the right properties.
\end{proof}

Theorem~\ref{thm:Schippers} now follows immediately from Corollary~\ref{cor:Schippers} and the following:

\begin{prop}
Let $f\in zK[[z]]$ be unitary, $n=\itval(f)$. Then $\log\, [f]^\Omega\in\frak m_K^n$. 
\end{prop}
\begin{proof}
We define a sequence $(h_j)_{j\geq 1}$ recursively as follows: set $h_1:=0$, and assuming inductively that $h_2,\dots,h_{j}$ have been defined already, where $j>0$, let $h_{j+1}:=(f_{j+1}-P^\Omega_{j+1}(h_2,\dots,h_{j}))\Omega_{j+1}$. Let $h:=\sum_{j\geq 1} h_jz^{j}\in z^{n+1}K[[z]]$ and $M:=\langle h\rangle^{\Omega}$. Then by the corollary above,
we have $(\exp M)_{1j}=f_j$ for every $j$.
Corollary~\ref{cor:Schippers} now yields $\exp M = [f]^\Omega$ and hence $\log\,[f]^\Omega=M=\langle h\rangle^{\Omega}\in \frak m_K^n$.
\end{proof}

\begin{remark}
The mistake in \cite[Lemma~3.10]{Schippers} pointed out in the example following the proof of Lemma~\ref{lem:3.10 corrected} affects the statements of items 3.14 and 3.15 and the proofs of 3.13--3.17  in loc.~cit.~(which concern the shape of $\log\, [f]$ for non-unitary $f\in z\C[[z]]$); however, based on the correct formula in Lemma~\ref{lem:3.10 corrected} above, it is routine to make the necessary changes. For example, the corrected version of \cite[Corollary~3.14]{Schippers} states that (using our notation) for $h\in z\C[[z]]$ and $j\geq 2$ we have
$$[\exp\langle h\rangle^{\mathbf 1}]_{1j} = \frac{h_{j}}{j-1}\left(\frac{e^{jh_1}-e^{h_1}}{h_1}\right)+\Phi_j(h_1,\dots,h_{j-1})$$
where $\Phi_j$ is an entire function $\C^{j-1}\to\C$.
\end{remark}

\section{The Iterative Logarithm}\label{sec:itlog}

\noindent
In this section we let $K$ be an integral domain which contains $\Q$ as a subring, and $\Omega$ be a reference sequence.  
Let $f\in zK[[z]]$ be unitary.
By Theorem~\ref{thm:Schippers} there exists a (unique) power series $h\in z^2K[[z]]$ such that $\log\,[f]^\Omega=\langle h\rangle^\Omega$. 
The identities \eqref{eq:conjugation logarithm}, \eqref{eq:convert iteration matrices, 2} and \eqref{eq:conjugate h} show that $h$ does not depend on $\Omega$. 
Indeed, we have
\begin{multline*}
h = \sum_{n\geq 1} \frac{(-1)^{n-1}}{n}h[n]\\ \text{where $h[0]=z$ and $h[n+1]=h[n]\circ f-h[n]\in z^{n+1}K[[z]]$ for every $n$.}
\end{multline*}
As in \cite{Ecalle2},
we call the power series $h$  the {\bf iterative logarithm of $f$}, and we denote it by $h=\itlog(f)$ or $h=f_*$. In the following we let $s$, $t$ be new distinct indeterminates, and we write 
$$f^{[t]} = \sum_{j\geq 1}\, (\exp\,t\langle f_*\rangle^\Omega)_{1j}\Omega_j\, z^j \in z+z^{n+1}K[t][[z]], \qquad n=\itval(f).$$
Note that $f^{[t]}$ does not depend on the choice of reference sequence $\Omega$.
For an element $a$ of a ring extension $K^*$ of $K$ let 
$$f^{[a]}:=f^{[t]}\big|_{t=a}\in z+z^{n+1}K^*[[z]],$$
so $f^{[0]}=z$ and $f^{[1]}=f$.
The notations $f^{[t]}$ and $f^{[a]}$ do not conflict with the ones introduced in Proposition~\ref{prop:EJ}:
by \eqref{eq:exp for commuting matrices} and \eqref{eq:Loewner, 2} (in Proposition~\ref{prop:Schippers}) we have
$$[f^{[s+t]}]^\Omega = \exp\,(s+t)\langle h\rangle^\Omega = \exp\,s\langle h\rangle^\Omega \cdot \exp\,t\langle h\rangle^\Omega=[f^{[s]}]^\Omega\cdot [f^{[t]}]^\Omega=[f^{[s]}\circ f^{[t]}]^\Omega$$
and hence 
\begin{equation}\label{eq:fractional iterates}
f^{[s+t]} = f^{[s]}\circ f^{[t]}
\end{equation}
in $K[s,t][[z]]$. Equation \eqref{eq:Loewner} also yields
$$\itlog(f)=\left.\frac{\partial f^{[t]}}{\partial t}\right\lvert_{t=0}.$$
%Using \eqref{eq:Mn} we obtain a result proved  in \cite{J1} by direct computation, without the convenient use of the matrix exponential and logarithm as above:
%\begin{cor}
%Let $h=f_*=\sum_{j\geq 2} h_{j}z^{j}$. Then $f^{[t]}=\sum_{j\geq 1} f^{[t]}_j z^j$ where
%$$f^{[t]}_j = \sum_{k_1+\cdots+k_n=j-1} h_{k_1+1}\cdots h_{k_n+1}\,i(i+k_1)\cdots (i+k_1+\cdots+k_{n-1})\frac{t^n}{n!}.$$
%\end{cor}
 If $a\in K$ then $(f^{[a]})^{[t]}=f^{[at]}$ by the uniqueness statement in Proposition~\ref{prop:EJ} and hence
\begin{equation}\label{eq:itlog powers}
\itlog(f^{[a]})=a\itlog(f)\qquad\text{for all $a\in K$.}
\end{equation}
%By \eqref{eq:exp for commuting matrices} we see that for unitary $f$, $g$ which commute (i.e., $f\circ g=g\circ f$) one has
%$$\itlog(f+g)=\itlog(f)+\itlog(g).$$
Acz\'el \cite{Aczel} and Jabotinsky \cite{J1} also showed that the iterative logarithm satisfies a functional equation (although \cite{Gronau} suggests that Frege had already been aware of this equation much earlier):

\begin{prop}[Acz\'el and Jabotinsky]\label{prop:Jab}
\begin{equation}\label{eq:Jab1}
f_* \cdot \frac{\partial f^{[t]}}{\partial z}= \frac{\partial f^{[t]}}{\partial t} = f_*\circ f^{[t]}
\end{equation}
and hence
\begin{equation}\label{eq:Jab2}
f_*\cdot \frac{df}{dz} = f_*\circ f.
\end{equation}
\end{prop}

%Here and below, we denote the derivation $\frac{d}{dz}$ on $K[[z]]$ by a prime. 
The equation \eqref{eq:Jab2} is known as {\it Julia's equation} in iteration theory. (See \cite[Section~8.5A]{KCG}.)
The first equation in \eqref{eq:Jab1} is simply \eqref{eq:Loewner}.
To show the second equation $\frac{\partial f^{[t]}}{\partial t} = f_*\circ f^{[t]}$, simply
differentiate \eqref{eq:fractional iterates} with respect to $s$:
$$\left.\frac{\partial f^{[u]}}{\partial u}\right|_{u=s+t}=\left.\frac{\partial f^{[u]}}{\partial u}\right|_{u=s+t}\cdot \frac{\partial(s+t)}{\partial s}=\frac{\partial f^{[s+t]}}{\partial s}=\frac{\partial (f^{[s]}\circ f^{[t]})}{\partial s} = \frac{\partial f^{[s]}}{\partial s}\circ f^{[t]}.$$
Setting $s=0$ yields the desired result.

\medskip

Suppose now that $K=\C$. Even if $f$ is convergent,  for given $a\in\C$ the formal power series $f^{[a]}$ is not necessarily convergent. In fact, by remarkable results of Baker \cite{Baker}, \'Ecalle \cite{Ecalle} and Liverpool \cite{Liverpool}, there are only three possibilities:
\begin{enumerate}
\item $f^{[a]}$ has radius of convergence $0$ for all $a\in\C$, $a\neq 0$;
\item there is some non-zero $a_1\in\C$ such that $f^{[a]}$ has positive radius of convergence if and only if $a$ is an integer multiple of $a_1$; or
\item $f^{[a]}$ has positive radius of convergence for all $a\in\C$.
\end{enumerate}
If (3) holds, then one calls $f$ {\bf embeddable} (in a continuous group of analytic iterates of $f$). This is a very rare circumstance; for example, Baker \cite{Baker-64} and Szekeres \cite{Szekeres} showed that if $f$ is the Taylor series at $0$ of a meromorphic function on the whole complex plane which is regular at $0$, then $f$ is not embeddable except in the case where
$$f=\frac{z}{1-cz} \qquad (c\in\C).$$
In this case, $\itlog(f)=cz^2$ by Example~\ref{ex:embeddable} and \eqref{eq:itlog powers}. 
Erd{\H{o}}s and Jabotinsky \cite{EJ} showed that in general, $f$ is embeddable if and only if $f_*=\itlog(f)$ has a positive radius of convergence. 
(See also \cite[Theorem~9.15]{Kuczma} or \cite{Scheinberg} for an exposition.)
As a consequence, very rarely does $f_*$ have a positive radius of convergence. (However,  \'Ecalle \cite{Ecalle3} has shown that $f_*$ is always Borel summable.) In particular, we obtain a negative answer to
the question posed in \cite[Question~4.3]{Schippers}: {\it if $f$ is convergent, is $f_*$ convergent?} Contrary to what is conjectured in \cite{Schippers},
the converse question (Question~4.1 in \cite{Schippers}), however, is seen to have a positive answer: {\it if $f_*$ is convergent, then $f$ is convergent.} 

\medskip

In the next section we  discuss when iterative logarithms satisfy algebraic differential equations.

\section{Differential Transcendence of Iterative Logarithms}\label{sec:difftr}

\noindent
Before we state the main result of this section, we introduce basic terminology concerning differential rings and differential polynomials.

\subsection*{Differential rings}
Let $R$ be a {\bf differential ring}, that is, a commutative ring $R$ equipped with a derivation $\der$ of $R$. We also write $y'$ instead of $\der(y)$ and similarly $y^{(n)}$ instead of $\der^{n}(y)$, where $\der^{n}$ is the $n$th iterate of $\der$. The set $C_R:=\{y\in R:y'=0\}$ is a subring of $R$, called the ring of constants of $R$.
A subring of  $R$ which is closed under $\der$ is called a {\bf differential subring} of $R$. 
If $R$ is a differential subring of a differential ring $\tilde R$ and $y\in \tilde R$, the smallest differential subring of $\tilde R$ containing $R\cup\{y\}$ is the subring $R\{y\}:=R[y,y',y'',\dots]$ of $\tilde R$ generated by $R$ and all the derivatives $y^{(n)}$ of $y$.
A differential field is a differential ring whose underlying ring happens to be a field. The ring of constants of a differential field $F$ is a subfield of $F$.
The derivation of a differential ring whose underlying ring is an integral domain extends uniquely to a derivation of its fraction field, and we always consider the derivation extended in this way. %If $F$ is a differential subfield of a differential field $\tilde F$ and $y\in \tilde F$, the smallest differential subfield of $\tilde F$ containing $F\cup\{y\}$ is the fraction field $F\langle y\rangle :=F(y,y',y'',\dots)$ of $F\{y\}$. 
If $R$ is a differential subring of a differential field $F$ and $y\in F^\times$, then $R_y:=\{a/y^n:a\in R,\ n\geq 0\}$ is a differential subring of $F$.

\subsection*{Differential polynomials}
Let $Y$ be a
differential indeterminate over the differential ring $R$. Then $R\{Y\}$ denotes the ring of
differential polynomials in $Y$ over $R$. As ring, $R\{Y\}$
is just the polynomial ring $R[Y,Y',Y'',\dots]$ in the distinct
 indeterminates
$Y^{(n)}$ over $R$, where as usual we write $Y=Y^{(0)}$, $Y'=Y^{(1)}$, $Y''=Y^{(2)}$. 
We consider
$R\{Y\}$ as the differential ring whose derivation,
extending the derivation of $R$ and also denoted by
$\der$, is given by $\der(Y^{(n)})=Y^{(n+1)}$
for every $n$. 
For $P(Y)\in R\{Y\}$ and $y$ an element of a differential ring containing $R$ as a differential subring, we let $P(y)$ be the element of that extension obtained
by substituting $y,y',\dots$ for $Y,Y',\dots$ in $P$,
respectively. 
We call an equation of the form
$$P(Y)=0\qquad\text{(where $P\in R\{Y\}$, $P\neq 0$)}$$ 
an {\bf algebraic differential equation} (ADE) over $R$, and a solution of such an ADE is an element $y$ of a differential ring extension of $R$ with $P(y)=0$.
We say that an element $y$ of a differential ring extension of $R$
is {\bf differentially algebraic over $R$}  if $y$ is the solution of an ADE over $R$, and if $y$ is not differentially algebraic over $R$, then $y$ is said to be {\bf differentially transcendental over $R$.}
Clearly to be algebraic over $R$ means in particular to be differentially algebraic over $R$.

Being differentially algebraic is transitive; this well-known fact  follows from basic properties of transcendence degree of field extensions:

\begin{lemma}\label{lem:transitivity of DA}
Let $F$ be a differential field and let $R$ be a differential subring of $F$. If $f\in F$ is differentially algebraic over $R$ and
$g\in F$ is differentially algebraic over $R\{f\}$, then $g$ is differentially algebraic over $R$.
\end{lemma}

\subsection*{Differential transcendence of iterative logarithms}
Let now $K$ be an integral domain containing $\Q$ as a subring, and let $z$ be an indeterminate over $K$.  We view $K[[z]]$ as a differential ring with the derivation $\frac{d}{dz}$. The ring of constants of $K[[z]]$ is $K$. We simply say that $f\in K[[z]]$ is {\bf differentially algebraic} or {\bf differentially transcendental} if $f$ is differentially algebraic  respectively differentially transcendental over $K[z]$. 
If $f\in K[[z]]$ is differentially algebraic, then $f$ is actually differentially algebraic over $K$, by Lemma~\ref{lem:transitivity of DA}.

As above, we let $t$ be a new indeterminate over $K$, and $K^*=K[t]$.
The goal of this section is to show:

\begin{theorem}\label{thm:BR}
Let $f\in zK[[z]]$ be unitary. Then $f_*\in z^2K[[z]]$ is differentially algebraic if and only if $f^{[t]}\in zK^*[[z]]$ is differentially algebraic, if and only if $f^{[t]}$ is differentially algebraic over $K^*$.
\end{theorem}

Before we give the proof, we introduce some more terminology concerning differential polynomials, and we make a few observations about how the derivation $\frac{d}{dz}$ of $K[[z]]$ and composition in $K[[z]]$ interact with each other, in particular in connection with solutions of Julia's equation. 
 
\subsection*{More terminology about differential polynomials}
Let $R$ be a differential ring and $P\in R\{Y\}$. The smallest $r\in\N$ such that 
$P\in R[Y,Y',\dots,Y^{(r)}]$ is called the
{\bf order} of the differential polynomial $P$. %, and the {\bf degree} $\deg P$ of $P$ is its (total) degree as an element of the polynomial ring $R[Y,Y',\dots]$ (with $\deg 0= -\infty$). 
Given a non-zero $P\in R\{Y\}$ we define its {\bf rank} to be the pair $(r,d)\in\N^2$ where $r=\order(P)$ and $d$ is the degree of $P$ in the indeterminate $Y^{(r)}$. In this context we order $\N^2$ lexicographically.

For any 
$(r+1)$-tuple $\bi =(i_0,\dots,i_r)$ of natural
numbers and $Q\in R\{Y\}$, put 
$$Q^{\bi} := Q^{i_0}(Q')^{i_1}\cdots (Q^{(r)})^{i_r}.$$
In particular, $Y^{\bi} = Y^{i_0}(Y')^{i_1}\cdots (Y^{(r)})^{i_r}$, and
$y^{\bi} = y^{i_0}(y')^{i_1}\cdots (y^{(r)})^{i_r}$ for $y\in R$.

Let $P\in R\{Y\}$ have order $r$, and let $\bi=(i_0,\dots,i_r)$ range over $\N^{1+r}$.
We denote by $P_{\bi}\in R$ the coefficient of $Y^{\bi}$ in $P$; then
$$P(Y)=\sum_{\bi} P_{\bi}\,Y^{\bi}.$$
We also define the support of $P$ as
$$\supp P := \big\{\bi : P_{\bi}\neq 0\big\}.$$
We set
$$\abs{\bi}:=i_0+\cdots+i_r,\qquad \dabs{\bi}:=i_1+2i_2+\cdots+ri_r.$$
For non-zero $P\in R\{Y\}$ we call 
$$\deg(P)=\max_{\bi\in\supp P} \abs{\bi},\qquad \wt(P)=\max_{\bi\in\supp P} \dabs{\bi}$$
the {\bf degree} of $P$ respectively {\bf weight} of $P$. We say that $P$ is {\bf homogeneous} if $\abs{\bi}=\deg(P)$ for every $\bi\in\supp P$ and {\bf isobaric} if $\dabs{\bi}=\wt(P)$ for every $\bi\in\supp P$. 

\subsection*{Transformation formulas}
Let $X$ be a differential indeterminate over $K[[z]]$.
An easy induction on $n$ shows that for each $n>0$ there are differential polynomials $G_{mn}\in\Z\{X\}$ \textup{(}$1\leq m\leq n$\textup{)} such that for all 
$f\in zK[[z]]$ and $h\in K[[z]]$ we have
$$(h^{(n)}\circ f)\cdot (f')^{2n-1}=G_{1n}(f)\,(h\circ f)'+G_{2n}(f)\,(h\circ f)''+\cdots+G_{nn}(f)\, (h\circ f)^{(n)}.$$
Moreover, $G_{mn}$ has order $n-m+1$, and is homogeneous of degree $n-1$ and isobaric of weight $2n-m-1$.
Set $G_{mn}:=0$ if $m>n$ or $m=0<n$, and $G_{00}:=(X')^{-1}\in\Z\{X\}_{X'}$.
Then the $G_{mn}$ satisfy the recurrence relation
$$G_{m,n+1} = (1-2n)G_{mn}X''+(G_{mn}'+G_{m-1,n})X'\qquad (m>0).$$
Organizing the $G_{mn}$ into a triangular matrix we obtain:
\begin{equation}\label{eq:G}
G:=(G_{mn})_{m,n}=
\begin{pmatrix}
(X')^{-1} & 0   & 0     & 0   				& \cdots \\
  & 1 	& -X''  & 3(X'')^2-X'X^{(3)} & \cdots \\
  &     & X' 	& -3X'X'' 			& \cdots \\
  &     &       & (X')^2      		& \cdots \\
  &     &       &            		& \ddots 
\end{pmatrix}.
\end{equation}
Note that $G_{nn}=(X')^{n-1}$ for every $n$. Now set
$$H_{kn} = \sum_{m=k}^n {m\choose k}X^{(m-k+1)}G_{mn}\in\Z\{X\}\qquad\text{for $k=0,\dots,n$.}$$
So if we define the triangular matrix
\begin{multline*}
B:=(B_{km})=\begin{pmatrix}
X' 	& X''   	& X^{(3)}	& X^{(4)}  	& \cdots \\
  	& X'		& 2X''  		& 3X^{(3)} 	& \cdots \\
  	&     	& X' 		& 3X'' 		& \cdots \\
  	&     	&       		& X'	  		& \cdots \\
  	&     	&       		&       		& \ddots 
\end{pmatrix} \\
\text{where $B_{km}={m\choose k}X^{(m-k+1)}$ for $m\geq k$,}
\end{multline*}
then 
\begin{multline*}%\label{eq:H}
H:=(H_{kn})=B\cdot G= \\ \begin{pmatrix}
1 	& X''  	& X'X^{(3)}-(X'')^2	& (X')^2X^{(4)}-4X'X''X^{(3)}+3(X'')^3   				& \cdots \\
  	& X' 	& X'X''  			& -3X'(X'')^2+2(X')^2X^{(3)}& \cdots \\
  	&     	& (X')^2 			& 0			& \cdots \\
  	&     	&       				& (X')^3      		& \cdots \\
  	&     	&       				&            		& \ddots 
\end{pmatrix}.
\end{multline*}
Each differential polynomial $H_{kn}$ has order at most $n-k+1$, and if non-zero, is homogeneous of degree $n$ and isobaric of weight $2n-k$. Note that for $n>0$, $H_{0n}$ has the form
$$H_{0n} = \sum_{m=1}^n X^{(m+1)}G_{mn} = (X')^{n-1}X^{(n+1)}+H_n \qquad\text{where $H_n\in\Z[X',\dots,X^{(n)}]$;}$$
in particular $\order(H_{0n})=n+1>\order(H_{kn})$ for $k=1,\dots,n$.

\medskip

Let now $f\in zK[[z]]$ and $h\in K[[z]]$ satisfy Julia's equation 
$$h\cdot f' = h\circ f.$$
We assume $f\neq 0$ (and hence $f'\neq 0$).
Then for every $n$:
$$(h^{(n)}\circ f)\cdot (f')^{2n-1} = H_{0n}(f)\,h+H_{1n}(f)\,h'+\cdots+H_{nn}(f)\,h^{(n)}.$$
%Let now $F=K((z))$ be the fraction field of $K[[z]]$ (the field of Laurent series in $z$ with coefficients from $K$).
Let $R:=K\{X\}_{X'}$, and denote
the $R$-algebra automorphism of $R\{Y\}$ with 
$$Y^{(n)}\mapsto (X')^{1-2n}\big(H_{0n}\,Y+H_{1n}\,Y'+\cdots+H_{nn}\,Y^{(n)}\big)\qquad\text{for every $n$}$$
also by $H$.
Then for every $P\in K\{Y\}$ we have
$$P(h)\circ f = H(P)\big|_{X=f,Y=h}.$$
Note that for every $i\in\N$ and $n$ we can write
\begin{multline*}
(X')^{(2n-1)i}\cdot H\big((Y^{(n)})^i\big)=(X')^{i(n-1)}Y^i(X^{(n+1)})^i+a_i\\
\text{where $a_{i}\in\Z[X',\dots,X^{(n+1)},Y,Y',\dots,Y^{(n)}]$ with $\deg_{X^{(n+1)}} a_i<i$.}
\end{multline*}
Hence given $\bi=(i_0,\dots,i_r)\in\N^{r+1}$, setting $d=\abs{\bi}$ and $w=\dabs{\bi}$, we may write
\begin{multline*}
(X')^{2w-d}\cdot H(Y^{\bi}) =  (X')^{w-d}(X')^{\mathbf i} Y^d + a_{\bi} \\
\text{where $a_{\bi}\in\Z[X',\dots,X^{(r+1)},Y,Y',\dots,Y^{(r)}]$ with $\deg_{X^{(r+1)}} a_{\bi}<i_r$.}
\end{multline*}

\subsection*{Proof of Theorem~\ref{thm:BR}}
Let $f\in zK[[z]]$ be unitary.
Suppose first that $f^{[t]}$ is differentially algebraic over $K^*$. Let  $P\in K^*\{Y\}$ be non-zero of lowest rank such that  $P(f^{[t]})=0$.
Differentiating  with respect to $t$ on both sides of this equation yields
$$P^*(f^{[t]})+\sum_{i=0}^r \frac{\partial P}{\partial Y^{(i)}}(f^{[t]})\cdot \frac{\partial (f^{[t]})^{(i)}}{\partial t}=0.$$
Here $r=\order(P)$ and
$P^*(Y)\in K^*\{Y\}$ is the differential polynomial obtained by applying $\frac{d}{dt}$ to each coefficient of the differential polynomial $P$. Now by Proposition~\ref{prop:Jab} we further have
$$\frac{\partial (f^{[t]})^{(i)}}{\partial t} = \left(\frac{\partial f^{[t]}}{\partial t}\right)^{(i)}=\big(f_*\cdot (f^{[t]})'\big)^{(i)}=\sum_{j=0}^i {j\choose i}(f^{[t]})^{(i-j+1)}f_*^{(j)}.$$
Since $\frac{\partial P}{\partial Y^{(r)}}$ has lower rank than $P$, by choice of $P$ we have $\frac{\partial P}{\partial Y^{(r)}}(f^{[t]})\neq 0$.
Hence $f_*$ satisfies a non-trivial (inhomogeneous) linear differential equation with coefficients from $K^*\{f^{[t]}\}$, and so
by Lemma~\ref{lem:transitivity of DA}, is differentially algebraic over $K^*$. Specializing $t$ to a suitable rational number in an ADE over $K^*$ satisfied by $f_*$ shows that then $f_*$ 
also satisfies an ADE over $K$, that is, $f_*$
is differentially algebraic over $K$.
 
Conversely, suppose that $f_*$ is differentially algebraic. Let $P\in K\{Y\}$ be non-zero, of some order $r$,  such that
$P(f_*)=0$.
Then $$H(P)(f^{[t]},f_*)=P(f_*)\circ f^{[t]}=0.$$
Let $d=\deg_{Y^{(r)}} P$. By the remarks in the previous subsection, for sufficiently large $N\in\N$ we have
\begin{multline*}
(X')^N \, H(P) = \sum_{\bi:i_r=d} P_{\bi}\,(X')^{N-\dabs{\bi}}Y^{\abs{\bi}}+A \\
\text{where $A\in K[X',\dots,X^{(r+1)},Y,Y',\dots,Y^{(r)}]$ with $\deg_{X^{(r+1)}}A<d$.}
\end{multline*}
For such $N$, the differential polynomial $$Q(X):=(X')^N\,H(P)\big|_{Y=f_*}\in R\{X\}$$ is non-zero, where $R=K\{f_*\}$, and satisfies $Q(f^{[t]})=0$.
Thus $f^{[t]}$ is differentially algebraic over $R$ and hence (by Lemma~\ref{lem:transitivity of DA}) over $K$, as required. \qed
 
\medskip

Let  $\mathcal F$ be a family of elements of $K[[z]]$.
Following \cite{BR} we say that $\mathcal F$ is {\bf coherent} if there is a non-zero differential polynomial $P\in K[z]\{Y\}$ such that $P(f)=0$ for every $f\in\mathcal F$.
If $\mathcal F$ is coherent, then $P$ with these properties may actually be chosen to have coefficients in $K$; see \cite[Lemma~2.1]{BR}. If $\mathcal F$ is not coherent, then we say that $\mathcal F$ is
{\bf incoherent}; 
we also say that $\mathcal F$ is {\bf totally incoherent} if every infinite subset of $\mathcal F$ is incoherent.
From the previous theorem we immediately obtain a result stated without proof in \cite{BR}:

\begin{cor}[Boshernitzan and Rubel \cite{BR}]\label{cor:BR}
Let $f\in zK[[z]]$ be unitary and let $\mathcal F:=\{f^{[0]},f^{[1]},f^{[2]},\dots\}$ be the family of iterates of $f$. Then exactly one of the following holds:
\begin{enumerate}
\item $f_*$ is differentially algebraic and $\mathcal F$ is coherent;
\item $f_*$ is differentially transcendental and $\mathcal F$ is totally incoherent.
\end{enumerate}
\end{cor}
\begin{proof}
By the theorem above, it suffices to show: if $f^{[t]}$ is differentially algebraic, then $\mathcal F$ is coherent, and if $f^{[t]}$ is differentially transcendental, then $\mathcal F$ is totally incoherent. The first implication is obvious (specialize $t$ to $n$ in a given ADE for $f^{[t]}$). For the second implication, suppose $\mathcal F$ is not totally incoherent. Then there exists an infinite sequence $(n_i)$ of pairwise distinct natural numbers  such that $\{f^{[n_i]}\}$ is coherent. Let $P\in K\{Y\}$, $P\neq 0$, be such that
$P(f^{[n_i]})=0$ for every $i$. With 
$g:=P(f^{[t]})\in K^*[[z]]$
we then have $g\big|_{t=n_i}=0$ for every $i$;  thus $g=0$  (since the coefficients of $g$ are polynomials in $t$ with coefficients from the integral domain $K^*$ of characteristic $0$). This shows that $f^{[t]}$ is differentially algebraic.
\end{proof}

\section{The Iterative Logarithm of $e^z-1$}\label{sec:proof}

\noindent
In this section we apply the results obtained in Sections~\ref{sec:lie} and \ref{sec:itlog} to the unitary power series $f=e^z-1\in z\Q[[z]]$.
Recall that the iteration matrix $[e^z-1]$ of this power series is the matrix $S=(S_{ij})\in 1+\tr_\Q^1$ consisting of the Stirling numbers  $S_{ij}={j\brace i}$  of the second kind
(cf.~\eqref{eq:Stirling itmatrix}).

\subsection*{Proof of the conjecture}
We first finish the proof of the conjecture stated in Section~\ref{sec:conj}. The matrix $S$ is related to $A=(\alpha_{ij})\in\tr^1_\Q$  via the equation
$$S^+ = \exp(A),$$
or equivalently (cf.~\eqref{eq:M+}):
$$A=\log(S)^+.$$
(Recall: for a given matrix $M=(M_{ij})\in\tr_\Q$ we defined $M^+=(M_{i+1,j+1})_{i,j}\in\tr_\Q$.)
The conjecture postulates the existence of a sequence $(c_n)_{n\geq 1}$ of rational numbers such that
\begin{equation}\label{eq:conj reformulated}
\alpha_{ij} = c_{j-i+1} {j+1 \choose i} \qquad\text{for $i<j$.}
\end{equation}
This now follows easily from the results of Section~\ref{sec:lie}:

\begin{prop}
Let $h=\itlog(e^z-1)\in z^2\Q[[z]]$, write $h=\sum_{n\geq 1} h_nz^n$ where $h_n\in\Q$, and define $c_n:=n!\,h_n$ for $n\geq 1$. Then \eqref{eq:conj reformulated} holds, and
$$c_{n}=\sum_{\substack{1\leq k<n\\ 1<n_1<\cdots<n_{k-1}<n_k=n}}  \frac{(-1)^{k+1}}{k}{n_2\brace n_1} {n_3\brace n_2}\cdots {n_k\brace n_{k-1}}$$ 
for every $n\geq 1$.
\end{prop}

\begin{proof}
We have $\log(S)=\langle h\rangle$ by Theorem~\ref{thm:Schippers}. Hence, using the formula for $\langle h\rangle_{ij}$ from Example~\ref{ex:factorials} we obtain  for $i<j$, as required:
$$\alpha_{ij}=\langle h\rangle_{i+1,j+1} = \frac{(j+1)!}{i!}h_{j-i+1} = \frac{(j+1)!}{i!(j-i+1)!} c_{j-i+1} =  c_{j-i+1} {j+1 \choose i}$$
The displayed identity for $c_n$ follows from $c_n=\langle h\rangle_{1n}=\log(S)_{1n}$.
\end{proof}

We note that the $c_n$ may also be expressed using the Stirling numbers of the first kind, using $\langle h\rangle=-\log(S^{-1})$:
$$c_n =  \sum_{\substack{1\leq k<n\\ 1<n_1<\cdots<n_{k-1}<n_k=n}}  \frac{(-1)^{k+n-n_1}}{k}{n_2\brack n_1} {n_3\brack n_2}\cdots {n_k\brack n_{k-1}}\qquad (n\geq 1).$$ 

\subsection*{Proof of the convolution identity}
We now turn to the convolution identity \eqref{eq:C} for Stirling numbers stated in the introduction. Jabotinsky's functional equation \eqref{eq:Jab2} for $f=e^z-1$, writing again $h=f_*$, reads as follows:
$$h\circ (e^z-1)=e^z\,h.$$
Taking derivatives on both sides of this equation and dividing by $e^z$ we obtain:
\begin{equation}\label{eq:functional equation for h'}
h'\circ (e^z-1)=h+h'.
\end{equation}
Now define, for $M\in 1+\tr^1_\Q$:
$$\Lambda(M):=\sum_{n} \frac{(-1)^n}{n+1} (M-1)^n\in 1+\tr^1_\Q,$$
so
\begin{equation}\label{eq:Lambda(M)}
\Lambda(M)\cdot (M-1)=\log(M).
\end{equation}
For later use we note that then for every $j\geq 1$:
\begin{multline}\label{eq:Lambda(M) row 1}
\sum_{k=1}^j \Lambda(M)_{1k} M_{k,j+1} = \sum_{k=1}^{j+1} \Lambda(M)_{1k} (M-1)_{k,j+1} = \\ (\Lambda(M)\cdot (M-1))_{1,j+1}=\log(M)_{1,j+1},
\end{multline}
where in the last equation we used \eqref{eq:Lambda(M)}.

\medskip

Taking $M=S$ we compute
$$\log(S) = \begin{pmatrix}
0 &	0 & 0 & 0				& 0				& 0				& 0				& \cdots \\		
  & 0 & 1 & -\frac{1}{2}	& \frac{1}{2}  	& -\frac{2}{3}	& \frac{11}{12}	& \cdots \\
  &   & 0 & 3 			& -2  			& \frac{5}{2}  	& -4   			& \cdots \\
  &   &   & 0  			& 6  			& -5  			& \frac{15}{2}	& \cdots \\
  &   &   &    			& 0  			&  10 			& 10  			& \cdots \\
  &   &   &    			&    			&  0 			& -15 			& \cdots \\
  &   &   &    			&    			&    			& 0   			& \cdots \\
  &   &   &    			&    			&    			&     			& \ddots  
\end{pmatrix}$$
and
$$\Lambda(S) = \begin{pmatrix}
1 	& 0 & 0				& 0				& 0				& 0				& \cdots \\		
	& 1 & -\frac{1}{2}	& \frac{1}{2}  	& -\frac{2}{3}	& \frac{11}{12}	& \cdots \\
  	&   & 1 				& -\frac{3}{2}  	& \frac{5}{2}  	& -\frac{25}{6}  & \cdots \\
  	&   &    			& 1  			& -3  			& \frac{15}{2}	& \cdots \\
  	&   &    			&    			&  1 			& -5  			& \cdots \\
  	&   &    			&    			&    			& 1	 			& \cdots \\
  	&   &    			&    			&    			&     			& \ddots  
\end{pmatrix}$$
We observe that the first row of $\Lambda(S)$ agrees with the first row of $\log(S)$ shifted by one place to the left. (This is simply a reformulation of the formula \eqref{eq:C}.)

\begin{prop}
For every $j\geq 1$,
$$\Lambda(S)_{1j} = \log(S)_{1,j+1}.$$
\end{prop}
%(This proposition is simply a reformulation of the convolution formula \eqref{eq:C}.)

\begin{proof}
As observed in \eqref{eq:Lambda(M) row 1},
\begin{equation}\label{eq:C1}
\sum_{k=1}^j \Lambda(S)_{1k} {j+1\brace k}=c_{j+1}\qquad\text{for $j\geq 1$.}
\end{equation}
On the other hand, by \eqref{eq:functional equation for h'} we have $[h']\cdot S = [h+h']$; thus
$$\sum_{k=1}^{j+1} c_{k+1} {j+1\brace k} = \sum_{k=1}^{j+1}\ [h']_{1k} S_{k,j+1} = ([h']\cdot S)_{1,j+1} = [h+h']_{1,j+1} = c_{j+1}+c_{j+2}$$ 
and hence
\begin{equation}\label{eq:C2}
\sum_{k=1}^{j} c_{k+1} {j+1\brace k} = c_{j+1}\qquad\text{for $j\geq 1$.}
\end{equation}
An easy induction on $j$ using \eqref{eq:C1} and \eqref{eq:C2} now yields $\Lambda(S)_{1j}=c_{j+1}= \log(S)_{1,j+1}$ for each $j\geq 1$, as claimed.
\end{proof}

\subsection*{Differential transcendence of the egf of $(c_n)$}
It is easy to see that for $n>0$, the $n$th iterate $\phi^{[n]}$ of $\phi=e^z-1$ is a solution of an ADE over $\Q$
 of order $n$. However, it is well-known that $\phi^{[n]}$ does not satisfy an ADE over $\C[z]$ of order $<n$. (See, e.g., \cite[Corollary~3.7]{AvdD2}.)
The egf of the sequence $(c_n)$ is $\itlog\phi$,
hence from Corollary~\ref{cor:BR} we obtain the fact (mentioned in the introduction) that this egf is differentially transcendental. 
In fact, Bergweiler \cite{Bergweiler} showed the more general result that if $f$ is (the Taylor series at $0$ of) any transcendental entire function, then $\itlog(f)$ is differentially transcendental (equivalently, by Corollary~\ref{cor:BR}, the family of iterates of $f$ is totally incoherent).
Moreover, by the results quoted at the end of the previous section, $\itlog\phi$ is not convergent. (This can also be shown directly; cf.~\cite{Lewin}.) See \cite{A} for a proof of a common generalization of these two facts.

\bibliographystyle{amsplain}

\end{document}